\documentclass[11pt]{article}

\usepackage[a4paper,left=24mm,right=24mm,top=23mm,bottom=27mm,headheight=15pt]{geometry}
\usepackage[T1]{fontenc}
\usepackage[utf8]{inputenc}
\usepackage{lmodern}
\usepackage{microtype}
\usepackage{amsmath}
\usepackage{amssymb}
\usepackage{amsthm}
\usepackage{mathrsfs}
\usepackage{graphicx}
\usepackage{enumitem}
\usepackage{booktabs}
\usepackage{siunitx}
\usepackage{tabularx}
\usepackage{array}
\usepackage{xcolor}
\usepackage{caption}
\usepackage{titlesec}
\usepackage{fancyhdr}
\usepackage[hidelinks]{hyperref}

\definecolor{metadata}{gray}{0.30}

\hypersetup{
  pdftitle={Stability and Dual Valuation of Contingent Claims under Rockafellian Perturbations},
  pdfauthor={Wolfgang Breytmann, Julio Deride, and Nicolas Hernandez},
  pdfkeywords={contingent claims, Rockafellian relaxation, epi-convergence, stochastic optimization, convex duality, stability analysis}
}

\setlist[enumerate]{leftmargin=.5in}
\setlist[itemize]{leftmargin=.5in}
\titleformat{\section}
  {\normalfont\large\bfseries}
  {\thesection}{0.75em}{}
\titleformat{\subsection}
  {\normalfont\normalsize\bfseries}
  {\thesubsection}{0.75em}{}
\titleformat{\subsubsection}
  {\normalfont\normalsize\bfseries}
  {\thesubsubsection}{0.75em}{}
\titlespacing*{\section}{0pt}{3.0ex plus 1ex minus .2ex}{1.4ex}
\titlespacing*{\subsection}{0pt}{2.4ex plus .8ex minus .2ex}{1.0ex}
\titlespacing*{\subsubsection}{0pt}{2.0ex plus .6ex minus .2ex}{0.8ex}

\theoremstyle{plain}
\newtheorem{theorem}{Theorem}[section]
\newtheorem{proposition}[theorem]{Proposition}
\newtheorem{lemma}[theorem]{Lemma}
\theoremstyle{definition}
\newtheorem{definition}[theorem]{Definition}
\theoremstyle{remark}
\newtheorem{remark}[theorem]{Remark}

\DeclareMathOperator*{\argmax}{arg\,max}
\DeclareMathOperator*{\argmin}{arg\,min}

\def\F{\mathcal{F}}
\def\R{\mathbb R}
\def\N{\mathbb N}

\def\to{\rightarrow}
\def\.{\cdot}

\def\<{\left<}
\def\>{\right>}
\newcommand{\nup}{^{\nu}}

\newcounter{case}

\begin{document}
\thispagestyle{empty}

\begin{center}
\vspace*{-1.2em}
{\LARGE\bfseries
\begin{minipage}{0.94\textwidth}
\centering
Stability and Dual Valuation of Contingent Claims\\[0.18em]
under Rockafellian Perturbations
\end{minipage}
}

\vspace{1.1em}
\rule{\textwidth}{0.55pt}
\vspace{1.0em}

\begin{tabularx}{\textwidth}{@{}>{\centering\arraybackslash}X >{\centering\arraybackslash}X >{\centering\arraybackslash}X@{}}
{\large\bfseries Wolfgang Breytmann}\textsuperscript{1}
& {\large\bfseries Julio Deride}\textsuperscript{2}
& {\large\bfseries Nicol\'as Hern\'andez}\textsuperscript{1} \\[0.45em]
\href{mailto:wolfgang.breytmann@usm.cl}{\texttt{wolfgang.breytmann@usm.cl}}
& \href{mailto:julio.deride@uai.cl}{\texttt{julio.deride@uai.cl}}
& \href{mailto:nicolas.hernandezs@usm.cl}{\texttt{nicolas.hernandezs@usm.cl}}
\end{tabularx}

\vspace{0.85em}
{\small
\textsuperscript{1} Departamento de Matem\'atica, Universidad T\'ecnica Federico Santa Mar\'ia, Valpara\'iso, Chile\\[0.2em]
\textsuperscript{2} Facultad de Ingenier\'ia y Ciencias, Universidad Adolfo Ib\'a\~nez, Santiago, Chile
}

\vspace{0.8em}
{\small\color{metadata}\today}
\end{center}

\vspace{0.9em}

\noindent
\begin{minipage}{\textwidth}
\rule{\textwidth}{0.35pt}
\vspace{0.55em}

\noindent{\bfseries Abstract.}
We study the stability of solutions to the discrete-time contingent-claim
problem over a finite investment horizon when uncertainty is modeled by
random variables with finite discrete support. Our main contribution is
to use Rockafellian perturbations as a framework for this stability
analysis: we construct perturbations of the underlying probability
distribution, of the contingent claim, and of both jointly, and we establish
epi-convergence of the corresponding approximating Rockafellians for the
primal problem. The associated hypo-convergent approximations yield
stable dual problems which, in turn, imply convergence of the dual
variables, interpreted as shadow prices. This analysis reveals a
connection between the duality gap and the value of perfect information
and it provides conditions under which strong duality holds. We also
construct examples in which epi-convergence fails due to critical
scenarios with vanishing probabilities but unbounded impacts, illustrating
the boundary between well-behaved and ill-conditioned contingent-claim
problems.

\vspace{0.65em}

\vspace{0.35em}

\vspace{0.65em}
\noindent{\footnotesize\itshape Funding. W. Breytmann and N. Hern\'andez were partially supported by the Agencia Nacional de Investigaci\'on y Desarrollo (ANID) through FONDECYT Grant No.~11240944.}

\vspace{0.55em}
\rule{\textwidth}{0.35pt}
\end{minipage}

\vspace{1.15em}

\section{Introduction}

    Contingent claims are financial contracts whose terminal payoffs depend on the uncertain realization of future market states.
    The pricing of contingent claims plays a central role in modern financial economics and mathematical finance; see, for example,  \cite{karatzas1998methods,pennanen2012introduction}. Central to this field is the Fundamental Theorem of Asset Pricing (FTAP), which establishes a formal equivalence between the absence of arbitrage and the existence of an equivalent probability measure under which discounted asset prices exhibit martingale or supermartingale behavior (for instance, \cite{pulido2014fundamental,delbaen1999fundamental,bouchard2015arbitrage}). In this context, a classical problem in mathematical finance is that of a writer of a contingent claim seeking an optimal investment strategy to hedge future liabilities. This problem is naturally formulated within the framework of stochastic programming and convex duality \cite{king2002duality,pennanen2011convex,pinar2010expected,pinar2010gain,tian2006continuous}.
    However, classical models must often be extended to accommodate market frictions such as partial information, portfolio constraints, and short-selling prohibitions  \cite{Dahl,Dalang--Morton--Willinger}. For instance, regulatory measures or inherent market structures often prohibit short selling to stabilize markets during periods of financial distress, such as the temporary short-selling bans implemented across several European countries during the sovereign debt crisis of 2011 \cite{beber2021short}.
    Under no-short-selling  constraints, the condition of No Free Lunch with Vanishing Risk (NFLVR) is preserved only if there exists a measure under which the prices of restricted assets are supermartingales \cite{pulido2014fundamental,follmer1997optional}. This extension requires special duality frameworks to characterize  the dual solutions in incomplete or restricted markets \cite{Dahl,pennanen2011dual}.
    Despite the above, it is well known that stochastic optimization problems are inherently unstable. Stochastic optimization models are often ill-conditioned relative to the underlying uncertainty. Small errors in the probability distribution or minor shifts in the support of random variables can induce disproportionately large jumps in optimal solutions \cite{royset2024rockafellian,romisch2003stability}. This vulnerability is particularly acute when models are built from empirical data that may be corrupted by outliers or measurement errors \cite{antil2024rockafellian,antil2026risk,royset2024rockafellian}. In such contexts, traditional distributionally robust optimization focuses on pessimistic worst-case scenarios within an ambiguity set to provide performance guarantees \cite{antil2026risk,glasserman2014robust}.

    Rather than adopting a distributionally robust optimization framework, we employ the framework of Rockafellian perturbations \cite{royset2021good,rockafellar1998variational}. A Rockafellian function embeds the original optimization problem into a family of perturbed problems, anchored at a zero perturbation vector \cite{royset2021optimization,rockafellar1998variational,royset2024rockafellian}. While robust approaches restrict the feasible set to guard against model risk, this work utilizes an optimistic approach based on problem relaxation \cite{antil2026risk,antil2024rockafellian}. This means considering the best-case scenarios under the perturbation defined by the Rockafellian. This distributionally optimistic perspective allows the model to \emph{keep in play} decisions that might be optimal under the true, uncorrupted distribution, by relaxing constraints rather than tightening them, effectively mitigating the influence of adversarial data or endogenous outliers \cite{antil2026risk,royset2024rockafellian,deride2024approximations}.
    The stability of these optimistic formulations is established through variational analysis, primarily the concept of epi-convergence \cite{rockafellar1998variational,feinberg2022epi,royset2021optimization}. Epi-convergence ensures that, under tightness conditions, the minima and minimizers of approximating problems approach those of the actual problem \cite{royset2021optimization,royset2018approximations}. In this paper, we study the stability of solutions for the discrete-time contingent-claim problem over a finite horizon with finite discrete support. We extend this analysis to both primal and dual problems: the primal stability is examined through epi-convergence of Rockafellians, while the dual problem utilizes hypo-convergence to account for the maximization  \cite{deride2024approximations}. This dual approach is essential for quantifying shadow prices of information \cite{pennanen2018convex} and for studying the convergence behavior of the corresponding solutions.

    Our main contribution is methodological. We introduce the framework of Rockafellian perturbations into the contingent-claim pricing problem under portfolio constraints, and we establish a comprehensive stability analysis for both primal and dual formulations. We construct Rockafellian perturbations associated with changes in the underlying probability distribution, perturbations of the claim constraints, and joint perturbations of both components. For each setting, we establish the epi-convergence of the corresponding sequence of approximating Rockafellians.

    On the dual side, we derive explicit dual formulations, prove hypo-convergence of the associated dual functions, and study the stability of optimal dual solutions. This analysis provides an economic interpretation of the dual variables as shadow prices and reveals a connection between the duality gap and the value of perfect information through a wait-and-see perspective. In addition, we identify conditions under which strong duality holds and characterize situations where duality gaps arise.

    Beyond the positive stability results, we also investigate the limitations of the Rockafellian framework. In particular, we construct examples of unstable models in which epi-convergence fails, illustrating the boundary between well-behaved and ill-conditioned contingent-claim problems.

\subsection{Preliminaries}

We adopt the notation and terminology of \cite{royset2021optimization}. We work in the $n$-dimensional space $\mathbb{R}^n$, and write $x_j$ for the $j$-th component of a vector $x \in \mathbb{R}^n$. The usual inner product is denoted by $\langle \cdot, \cdot \rangle$. For any $m\in\N$, the simplex is denoted by $\Delta_{m-1} := \left \{p \in \R^m_+ : \sum_{j=1}^m p_j =1 \right \}$. The extended real line is denoted by $\bar{\mathbb{R}} := \mathbb{R} \cup \{-\infty, +\infty\}$. We adopt the standard conventions $\infty - \infty = \infty$ and $0 \cdot \infty = 0$. For a set $C \subset \mathbb{R}^n$, the \emph{indicator function} $\iota_C : \mathbb{R}^n \to \bar{\mathbb{R}}$ is defined by $\iota_C(x) = 0$ if $x \in C$ and $\iota_C(x) = \infty$ otherwise.

For a function $g : \mathbb{R}^n \to \bar{\mathbb{R}}$, its \emph{effective domain} and \emph{epigraph} are defined, respectively by
$\operatorname{dom} g := \{x \in \mathbb{R}^n \mid g(x) < \infty\}$, and $\operatorname{epi} g := \{(x,\alpha) \in \mathbb{R}^n \times \mathbb{R} \mid g(x) \le \alpha\}$.

We say that $g$ is \emph{proper} if $\operatorname{dom} g \neq \emptyset$ and $g(x) > -\infty$ for all $x \in \mathbb{R}^n$. We say that $g$ is \emph{lower semicontinuous (lsc)} if $\operatorname{epi} g$ is closed in $\mathbb{R}^n \times \mathbb{R}$. Equivalently, $g$ is lsc if and only if all its sublevel sets, $\{x \in \mathbb{R}^n : g(x) \le \alpha\}$, are closed for every $\alpha \in \mathbb{R}$. Another equivalent characterization is that $g$ is lsc if and only if, for every $x \in \mathbb{R}^n$ and every sequence $(x_k)_k \subset \mathbb{R}^n$ with $x_k \to x$, one has $ g(x) \le \liminf_{k \to \infty} g(x_k)$. We say that $g$ is convex if $\operatorname{epi} g$ is a convex set and $g$ is strongly convex, with modulus $\beta$, if the function $g(\cdot)-\frac{\beta}{2} \|\cdot\|^2$ is convex.

Consider a sequence of functions $\{g^\nu: \R^n \to \bar{\R},\nu\in\N\}$ and a function $g:\R^n\to\bar{\R}$.  We say that $g^\nu$ {\em epi-converges} to $g$, written  $g^\nu \xrightarrow{e} g$, if at every $x\in\R^n$, the following two conditions hold:
\begin{enumerate}
    \item \label{epi:liminf}
    For every $x \in \mathbb{R}^n$ and $(x^\nu)_{\nu \in \mathbb{N}} \subset \mathbb{R}^n$ such that $x^\nu \to x$, one has
   $\liminf_{\nu \to \infty} g^\nu(x^\nu) \geq g(x).$
    \item \label{epi:limsup}
    For every $x \in \mathbb{R}^n$, there exists $(x^\nu)_{\nu \in \mathbb{N}} \subset \mathbb{R}^n$ such that $x^\nu \to x$ and $\limsup_{\nu \to \infty} g^\nu(x^\nu) \leq g(x)$.
\end{enumerate}

We say that $g^\nu$ \emph{hypo-converges} to $g$, denoted by $g^\nu \xrightarrow{h} g$, if $-g^\nu \xrightarrow{e} -g$.

A sequence of functions $\{g^\nu : \mathbb{R}^n \to \overline{\mathbb{R}}\}$ is called \emph{tight} if for every $\varepsilon > 0$ there exist a compact set $B_\varepsilon \subset \mathbb{R}^n$ and an index $\nu_\varepsilon \in \mathbb{N}$ such that
\[
\inf_{x \in B_\varepsilon} g^\nu(x)
\le
\inf_{x \in \mathbb{R}^n} g^\nu(x) + \varepsilon,
\quad \text{for all } \nu \ge \nu_\varepsilon.
\]
Let $g : \mathbb{R}^n \to \overline{\mathbb{R}}$. The \emph{conjugate} and \emph{biconjugate} of $g$ are the functions $g^*, g^{**} : \mathbb{R}^n \to \overline{\mathbb{R}}$ defined by $g^*(y) := \sup_{x \in \mathbb{R}^n} \{ \langle x, y \rangle - g(x) \}$, and
$g^{**}(x) := \sup_{y \in \mathbb{R}^n} \{ \langle x, y \rangle - g^*(y) \}$.

We formulate our optimization problems using the concept of a Rockafellian function. Given a function $\varphi : \mathbb{R}^n \to \overline{\mathbb{R}}$, to be minimized, a function $f : \mathbb{R}^m \times \mathbb{R}^n \to \overline{\mathbb{R}}$ is called a \emph{Rockafellian} (for $\varphi$) if
$f(0,x) = \varphi(x)$, for all $x \in \mathbb{R}^n$.

For a fixed vector $y \in \mathbb{R}^m$, the \emph{tilted Rockafellian} associated with $f$ is defined by $f_y(u,x) := f(u,x) - \langle y, u \rangle$.  Similarly, for a sequence $\{y^\nu\}_{\nu \in \mathbb{N}} \subset \mathbb{R}^m$ and a sequence of functions $\{f^\nu\}$, the \emph{approximating tilted Rockafellian} is defined by $f^\nu_{y^\nu}(u,x) := f^\nu(u,x) - \langle y^\nu, u \rangle$.

The \emph{Lagrangian} $l : \mathbb{R}^n \times \mathbb{R}^m \to \overline{\mathbb{R}}$ associated with $f$ is defined by
\begin{equation}
    l(x,y) := -\bigl(f(\cdot,x)\bigr)^*(y)
    = \inf_{u \in \mathbb{R}^m} \{ f(u,x) - \langle y,u \rangle \}.
    \label{lagrangian}
\end{equation}
For a fixed multiplier $y \in \mathbb{R}^m$, the \emph{Rockafellian relaxation} of the problem $\min_{x \in \mathbb{R}^n} \varphi(x)$ is given by
\[
\min_{x \in \mathbb{R}^n} \, l(x,y).
\]
Optimization over the multiplier $y$ yields the dual problem. The \emph{dual objective function} $\psi : \mathbb{R}^m \to \overline{\mathbb{R}}$ is defined by
\[
\psi(y) := -f^*(y,0)
= \inf_{(u,x) \in \mathbb{R}^m \times \mathbb{R}^n} \{ f(u,x) - \langle y,u \rangle \},
\]
and the corresponding \emph{dual problem} is $\sup_{y \in \mathbb{R}^m} \psi(y)$.

When the function $f : \mathbb{R}^m \times \mathbb{R}^n \to \overline{\mathbb{R}}$ is convex, so is the following auxiliary function $\phi : \mathbb{R}^m \to \overline{\mathbb{R}}$
\[
\phi(u) := \inf_{x \in \mathbb{R}^n} f(u,x).
\]
Moreover, in terms of the biconjugate, we have
\[
\phi^{**}(0) = \liminf_{u \to 0} \phi(u) = \sup_{y \in \mathbb{R}^m} \psi(y).
\]

\section{Market model and investment problem}

We consider the optimization problem of an investor who seeks to choose a trading strategy that maximizes the expected terminal value of the portfolio, subject to the constraint of hedging the cash flows of a given contingent claim.

Let $(\Omega,\mathcal{A},\mathbb{P})$ be a probability space and let $T>0$ be a finite time horizon. We work in discrete time and consider an $\mathbb{R}^d$-valued stochastic process $\xi := \{\xi_t\}_{t=0}^{T+1}$, where $\xi_0$ is deterministic and, for each $t \ge 1$, $\xi_t : \Omega \to \mathbb{R}^d$ has finite support. The uncertainty in the model is described by the realizations of the process $\xi$, which we refer to as the environment process. This means that we endow the probability space with the filtration generated by $\xi$. We denote it by $\{\F_t\}_{t=0}^{T+1}$.

For $t\in\{0,\ldots,T+1\}$, we define the history of the environment process up to time $t$ by
\[
  \xi_{0:t}:= (\xi_0,\ldots,\xi_t) \in \mathbb{R}^{(t+1)d}.
\]

We denote by $\{S_t\}_{t=0}^{T+1}$ the process of market prices of the $k+1$ basic securities, that is,
\[
S_t(\xi_{0:t})=(S_t^0,S_t^1(\xi_{0:t}),...,S_t^k(\xi_{0:t})),
\]
where $S_t^0$ denotes the price of a risk--free asset, and for $i\in\{1,\ldots,k\}$, $S_t^i(\xi_{0:t})$ is the price of the $i$-th risky asset at time $t$. We assume that all asset prices are strictly positive. We also assume that the market is arbitrage-free, meaning that there exists no self-financing trading strategy with zero initial wealth that yields a nonnegative terminal wealth almost surely and a strictly positive terminal wealth with positive probability.

For each $t \in \{0,...,T\}$, let $G_t : \mathbb{R}^{(t+1)d} \to \mathbb{R}$ denote the payoff of the contingent claim at time $t$. The associated payoff process is therefore $G:=\{G_t(\xi_{0:t})\}_{t=0}^T$. We assume that $G_0 > 0$, representing the initial capital, which may consist of both borrowed funds and capital contributed by, and that $G_t < 0$ for $t \geq 1$ corresponding to the cash outflows required to meet the claim payments. The randomness of the contingent claim is entirely determined by the history of the environment process.

The writer sets up a portfolio to meet the contingent claim with some investment strategy. For every $t\in\{0,\dots,T\}$, let $X_t$ be the portfolio allocation at time $t$, which is a function of the environment process. Therefore, $X$ is adapted to the filtration generated by $\xi$. By letting $n_t\in\N$ denote the size of the support of $\xi_{0:t}$, we can represent $X_t$ as an element of $\R^{(k+1)n_t}$. We then denote $X:= (X_0,X_1,...,X_T)$ to the trading strategy over the whole time period. The writer seeks to maximize the expected terminal wealth of the portfolio by solving the following optimization problem, in dimension $n:=(k+1)N_T$, where $N_T:=\sum_{t=0}^T n_t$.
\begin{align}
\max_{X\in \R^n} \ &\mathbb{E}\langle S_{T+1}(\xi_{0:T+1}),X_T(\xi_{0:T}) \rangle \label{max_problem} \\
\text{s.t. } & \langle S_0,X_0\rangle \leq G_0, \notag  \\
& \langle S_t(\xi_{0:t}),X_t(\xi_{0:t}) -X_{t-1}(\xi_{0:t-1}) \rangle \leq G_t(\xi_{0:t}), \ \ \forall t\in\{1,...,T\}  \ \ \text{a.s.} \notag \\
&X_t(\xi_{0:t}) \geq 0, \ \forall t\in\{0,...,T\} \ \ \text{a.s.} \notag
\end{align}
We impose the No–short-selling constraint for technical reasons, namely to preserve important properties of the Rockafellian under perturbations of the probability measure, such as tightness, and to avoid degeneracy in the dual problem (see Remark~\ref{rem:positivity}). The investment strategy $X$ is thus a vector in the following feasible region, which is closed and convex.
\begin{align*}
\mathcal{X}:=\{  X\in \R^n :  ~~& \langle S_0,X_0\rangle \leq G_0, ~ X_t(\xi_{0:t}) \geq 0, \ \forall t=0,...,T \ \ \text{a.s.}, \\
 & \langle S_t(\xi_{0:t}),X_t(\xi_{0:t}) -X_{t-1}(\xi_{0:t-1}) \rangle \leq G_t(\xi_{0:t}), \ \ \forall t=1,...,T  \ \ \text{a.s.} \}.
\end{align*}
The last decision $X_T$ guarantees that a terminal wealth is computed from the portfolio position after paying the claim.

While standard portfolio hedging models in the literature typically liquidate the terminal portfolio at time T concurrent with the final claim payment, our framework extends the horizon to $T+1$. This choice reflects a scenario where the writer treats the contingent claim as an intermediate liability within a longer-term investment strategy. Because no claim payments occur at time $T+1$, the allocation $X_T$ chosen at time $T$ is completely unburdened by hedging constraints. This allows the writer to optimize the final period's risk-return profile purely for terminal wealth maximization and take more aggressive positions.

\begin{remark}\label{rem:compactness-X}
Observe that the feasible set $\mathcal X$ may be assumed to be compact without loss of generality, and we shall do so from now on. Indeed, at time $t=0$, for every asset $i$ such that $S_0^i>0$, the constraint $\langle S_0, X_0\rangle \leq G_0$ implies that $X_0^i\in[0,M_0^i]$ with $M_0^i:=G_0/S_0^i$. Proceeding inductively, at any $t>0$ and for any asset such that $S_t^i(\xi_{0:t})>0$, the constraint  $\langle S_t(\xi_{0:t}), X_t - X_{t-1}\rangle \leq G_t(\xi_{0:t})$ yields
\[
0\leq X_t^i(\xi_{0:t}) \leq \frac{G_t(\xi_{0:t}) + \langle S_t(\xi_{0:t}), M_{t-1}(\xi_{0:t-1})\rangle}{S_t^i(\xi_{0:t})} =: M_t^i(\xi_{0:t}).
\]
\end{remark}

\section{Rockafellian Perturbations and stability} \label{sec:3}

In the context of the optimal portfolio selection problem \eqref{max_problem}, we introduce its primal objective function. To simplify the notation, we no longer write the environment process $\xi$ and remind the reader that $S$ and $X$ are adapted to it. Written in minimization form, we define $\varphi:\R^n\rightarrow\overline\R$ by $\varphi(X) := -\mathbb{E} [\langle S_{T+1},X_T \rangle] + \iota_{\mathcal{X}}(X)$. We are then interested in the following optimization problem
    \begin{equation}
        \min_{X\in \R^n} \varphi (X)
        . \tag{P} \label{P}
    \end{equation}
Given that the environment process can follow finitely many trajectories, we can represent the uncertainty of the model through a tree graph, as illustrated in Figure~\ref{fig:initial-tree}. We add a node at layer $t$ for every realization of $\xi_t$ and an edge for every possible time-transition. Observe that an allocation decision is made at every non-terminal node.

\begin{figure}[t]
\centering
\includegraphics[width=.75\textwidth]{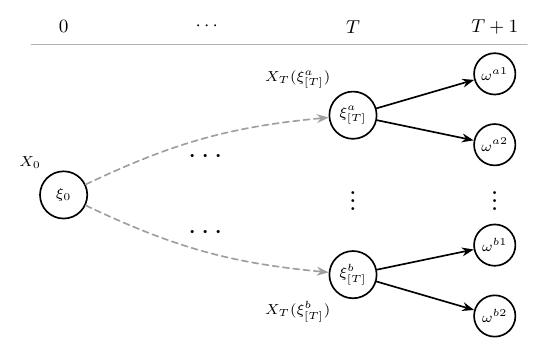}
\caption{Representative scenario tree for the environmental process. Each node at time $t$ corresponds to an information state, and allocation decisions are made at non-terminal nodes.}
\label{fig:initial-tree}
\end{figure}

To examine the stability of the model, we extend the tree just described. We introduce additional scenarios at time $T+1$, each representing an unforeseen `critical' event. In these scenarios the value of some, possibly all, risky assets at time $T$ is multiplied by a positive factor. To be more precise, for each added scenario $j$ originating from the parent node $j^-$, the asset $i\in\{1,\dots,k\}$ is scaled by a factor $\kappa_j^i > 0$.

Since, by construction, the terminal nodes of a scenario tree are associated with probabilities, the introduction of the new scenarios requires a consistent extension of the tree. To preserve the probabilistic structure, the newly added scenarios are assigned zero probability. They are thus part of the scenario structure, but they do not alter the original probabilities in the model. Naturally, when we perturb the problem, we will allow these nodes to have positive probabilities.

The augmented scenario tree is labeled in a top-down manner at time $T+1$. In Figure~\ref{fig:labeled-tree} we depict a situation in which one critical scenario is added originating from node $1^-$ and two critical scenarios are added originating from node $m^-$.

Let $m \in \N$ denote the total number of terminal nodes after augmentation, and let $p \in \Delta_{m-1}$ be the probability vector associated with these nodes. We denote by $I$ the set of indices corresponding to the original terminal nodes and by $\tilde I$ the set of indices associated with the newly added scenarios. Then, as already explained, we impose $p_j =0, \ \forall j \in \tilde I$ so that the original optimization problem \eqref{P} is not affected by the mere augmentation of the tree. The augmented tree is labeled top-down at time $T+1$, as depicted in Figure~\ref{fig:labeled-tree} for a case with one critical scenario added at node $1^-$ and two critical scenarios added at node $m^-$.

\begin{figure}[t]
\centering
\includegraphics[width=.75\textwidth]{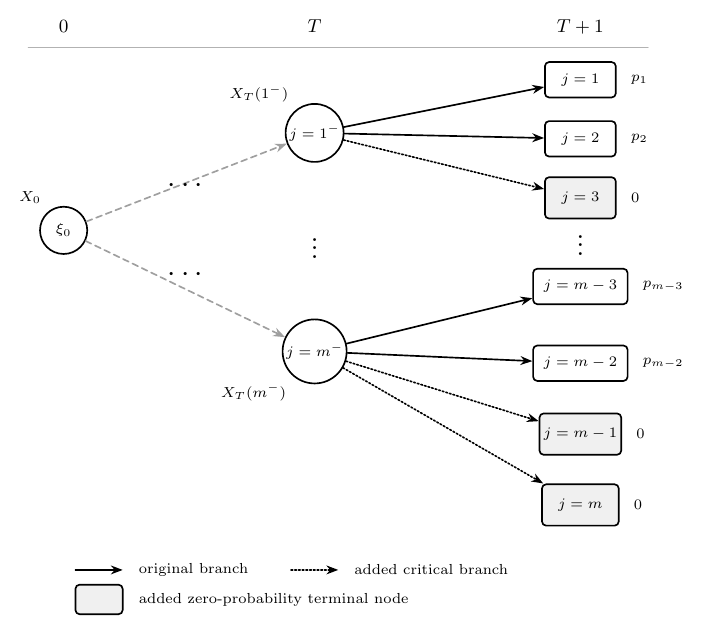}
\caption{Augmented scenario tree used in the stability analysis. Solid branches correspond to scenarios already present in the original model. Dotted branches represent critical scenarios added at time $T+1$ and assigned probability zero in the unperturbed model.}
\label{fig:labeled-tree}
\end{figure}

In the approximating scheme, however, the probabilities of the scenarios are allowed to change. We denote by $p\nup$, with $\nu\in\N$, an approximate probability vector, which may assign positive probabilities to the scenarios in $\tilde I$. In Section~\ref{sec:3.1}, we study stability properties via epi-convergence. Our motivation is that, in practice, the probability vector $p$ can be unknown; instead, one may only have access to an approximate distribution with an associated vector $p\nup$.

\begin{remark}
The probability vector $p^{\nu}$ may not accurately represent the true scenario distribution encoded by $p$. In particular, $p^{\nu}$ may assign positive probability to scenarios that are not originally present in the model (scenarios in $\tilde I$), while others that belong to the support of $p$ (scenarios in $I$) may have zero probability under $p^{\nu}$.
\end{remark}

\begin{remark}\label{rem:definition-p}
Formally, the vector $p\in\Delta_{m-1}$ is defined for every $i\in\{1,\dots,m\}$ by
\[
p_i := \mathbb{P}\big(\text{the realization of the environment process at time } T+1 \text{ corresponds to node } i\big).
\]
We identify the vector $p$ with the probability measure $\mathbb{P}$. Similarly, in the approximation schemes, each vector $p^{\nu}$ is identified with a probability measure $\mathbb{P}^{\nu}$. For notational convenience, expectations with respect to $\mathbb{P}$ and $\mathbb{P}^{\nu}$ are denoted by $\mathbb{E}^{p}$ and $\mathbb{E}^{p^{\nu}}$, respectively.
\end{remark}

\subsection{Rockafellian perturbation of the probability distribution}
\label{sec:3.1}

We start by proposing the Rockafellian (for $\varphi$) $f_1:\R^m \times \R^n \rightarrow \overline{\R}$, defined as follows
    \begin{equation}
     f_1(u,X) := -\sum_{j=1}^m (p_j+u_j) \langle S_{T+1}(j), X_T(j^-) \rangle + \iota_{\mathcal{X}}(X) + \frac{1}{2} \theta \|u\|^2_2 + \iota_{[0,+\infty[^{m}} (p+u), \label{R1}
    \end{equation}
where $p$ is the vector defined in Remark~\ref{rem:definition-p} and $\theta$ is a positive regularization parameter. The vector $u$ represents a perturbation in the probability distribution, so that the leftmost term is a perturbation of the expectation in \eqref{P}. Indeed, when $u=0$, we recover $\mathbb{E}^p[\langle S_{T+1},X_T \rangle]$.

Now let $p\nup \in \Delta_{m-1}$ be an approximation of $p$. We define the approximating objective function $\varphi\nup: \R^n \rightarrow \overline{\R}$ by
\begin{equation} \label{AP}
       \varphi^\nu (X) :=  -\mathbb{E}^{p^\nu}[\langle S_{T+1},X_T \rangle] + \iota_{\mathcal{X}}(X),
\end{equation}
and next the approximating Rockafellian (for $\varphi^\nu$) $f_1^{\nu}:\R^m\times\R^n \rightarrow \overline{\R} $ as
\begin{equation}
        f_1^{\nu}(u,X) :=-\sum_{j=1}^m (p_j \nup+u_j) \langle S_{T+1}(j), X_T(j^-) \rangle + \iota_{\mathcal{X}}(X) + \frac{1}{2} \theta^{\nu} \|u\|_2^2 + \iota_{[0,+\infty[^m} (p^\nu+u),
        \label{AR1}
    \end{equation}
where $\{\theta\nup\}_{\nu \in \N}$ is a family of positive parameters that determines the shape of the regularization term. We will also study a variation in which the perturbation must be negative, so we define
\begin{equation}\label{R1.1}
     \tilde f_1(u,X) := f_1(u,X) + \iota_{]-\infty,0]^{m}} (u),
\end{equation}
and its approximation
\begin{equation}\label{AR1.1}
     \tilde f_1^{\nu}(u,X) := f_1\nup(u,X)  + \iota_{]-\infty,0]^{m}} (u).
\end{equation}

\begin{proposition}\label{prop:epi1}
Let $\theta^{\nu} \to \theta$ and $p^{\nu} \to p$. Then the sequence of approximating Rockafellians $f_1^{\nu}$ epi-converges to $f_1$. Moreover, the sequence $\tilde f_1^{\nu}$ epi-converges to $\tilde f_1$.
\end{proposition}

\begin{proof}
(i) We show that Conditions \ref{epi:liminf} and \ref{epi:limsup} for epi-convergence hold. Let $(u,X)\in\R^m\times\R^n$ and $(u^\nu,X^\nu)\subset \R^m\times\R^n$ be such that $(u^\nu,X^\nu) \rightarrow (u,X)$. Since the sets $\mathcal{X}$ and $[0,+\infty[^m$ are closed, their indicator functions are lower semicontinuous. Then we have
$$
\liminf_{\nu \to\infty} ~ \iota_{[0,+\infty[^{m}}(p^\nu+u\nup) \geq \iota_{[0,+\infty[^{m}}(p+u), \quad \liminf_{\nu\to\infty} ~\iota_{\mathcal{X}}(X^\nu) \geq \iota_{\mathcal{X}}(X).
$$
Next, from the convergence of the sequences, we obtain
\begin{align*}
    \liminf_{\nu\to\infty} ~ \sum_{j=1}^m (p_j^{\nu}+u_j^{\nu}) \langle -S_{T+1}(j), X_T^{\nu}(j^-) \rangle &=
\sum_{j=1}^m \lim_{\nu\to\infty} ~ (p_j^{\nu}+u_j^\nu) \langle -S_{T+1}(j), X_T\nup(j^-) \rangle \\
    &= \sum_{j=1}^m (p_j+u_j) \langle-S_{T+1}(j), X_T(j^-)\rangle
\end{align*}
Combining all the facts above, we find that $\liminf_{\nu \to \infty} f_1^\nu(u^\nu,X^\nu) \geq f_1(u,X).$

Consider now $(u,X) \in \R^m \times \R^n$. Without loss of generality, assume that $p+u \in [0,+\infty[^m$ and $X \in \mathcal{X}$ so that $f_1(u,X)<\infty$. For every $\nu\in\N$ define $X\nup := X$ and $ u\nup := p+u-p\nup$ and note that $u\nup \rightarrow u$. Then
\begin{align*}
    \limsup_{\nu \to\infty} f_1\nup(u\nup,X\nup) &= \lim_{\nu \to\infty} \left ( \sum_{j=1}^m (p_j+u_j) \langle -S_{T+1}(j), X_T(j^-)\rangle + \frac{1}{2} \theta\nup \sum_{j=1}^{m} (u_j\nup)^2 \right ) = f_1(u,X).
\end{align*}
We thus conclude that $f_1^{\nu}$ epi-converges to $f_1$.

(ii) The epi-convergence of $\tilde{f}^\nu_1$ follows directly from \cite[Proposition 3.6(b)]{deride2024approximations}, by choosing the functions $g_0(X)=\iota_{\mathcal{X}}(X)$ and $g_i(X)=\langle -S_{T+1}(i), X_T(i^-) \rangle$.
\end{proof}

\begin{remark}
Observe that the previous epi-convergence is not affected by the market prices in the critical scenarios. By this we mean that for any values of the scale factors $\kappa_j^i$, epi-convergence holds. However, if we perturb those factors and the sequence $(\kappa^i_j)^\nu$ grows faster than the rate at which $p\nup$ converges to $p$, then epi-convergence can no longer be guaranteed. This fact is discussed in more detail in Subsection~\ref{sec:counterproposition}.
\end{remark}

\subsection{Rockafellian in claims}
\label{sec:3.2}
We now introduce a Rockafellian perturbation in the set of constraints $\mathcal{X}$ of Problem \eqref{P} and proceed to study similar stability properties as in the previous subsection.

For given functions $u_t : \mathbb{R}^{(t+1)d} \to \mathbb{R}$, with $t \in \{0,...,T\}$, we denote the associated stochastic process by $u:=\{u_t(\xi_{0:t})\}_{t=0}^T$. We define next the perturbed set as
\begin{align*}
\mathcal{X}(u):=\{  X\in \R^n :  ~~& \langle S_0,X_0\rangle \leq G_0+u_0, ~ X_t(\xi_{0:t}) \geq 0, \ \forall t=0,...,T \ \ \text{a.s.}, \\
 & \langle S_t(\xi_{0:t}),X_t(\xi_{0:t}) -X_{t-1}(\xi_{0:t-1}) \rangle \leq G_t(\xi_{0:t})+ u_t(\xi_{0:t}), \ \ \forall t=1,...,T  \ \ \text{a.s.} \}.
\end{align*}
We define the Rockafellian (for $\varphi$) $f_2: \R^{N_T}\times\R^n \rightarrow \overline{\R}$ by
\begin{equation}
    f_2(u,X)  := -\mathbb{E}^p \langle S_{T+1},X_T\rangle + \iota_{\mathcal{X}(u)}(X).
    \label{R2}
\end{equation}
We recall the vector $p$ defined in Remark~\ref{rem:definition-p} and, for an approximation $p^\nu$ of $p$, the function $\varphi\nup$ defined in \eqref{AP}. We define now the approximating Rockafellian (for $\varphi\nup$) $f_2\nup: \R^{N_T} \times \R^n \rightarrow \overline{\R}$ by

\begin{equation}
    f_2\nup (u,X) = -\mathbb{E}^{p\nup} \langle S_{T+1},X_T\rangle + \iota_{\mathcal{X}(u)}(X).
    \label{AR2}
\end{equation}
The next result is analogous to Proposition~\ref{prop:epi1}.

\begin{proposition}\label{prop:epi2}
    Let $p\nup \to  p$, then the sequence of approximating Rockafellians $f_2\nup$ epi-converges to $f_2$.
\end{proposition}

\begin{proof}
We show that Conditions \ref{epi:liminf} and \ref{epi:limsup} for epi-convergence hold. Let $(u,X)\in\R^{N_T} \times\R^n$ and $(u^\nu,X^\nu)\subset \R^{N_T}\times\R^n$ be such that $(u^\nu,X^\nu) \rightarrow (u,X)$, then $\liminf_{\nu \in \N} \iota_{\mathcal{X}(u\nup)}(X\nup) \geq \iota_{\mathcal{X}(u)}(X)$. This can be seen directly by replacing the term $\iota_{\mathcal{X}(u)}(X)$ by $\iota_{\tilde{\mathcal{X}}}(X,u)$ in the Rockafellian \eqref{R2}, where
\begin{align*}
\tilde{\mathcal{X}}:=\{  (X,u)\in \R^{n+N_T} :  ~& \langle S_0,X_0\rangle \leq G_0+u_0, ~ X_t(\xi_{0:t}) \geq 0, ~ \forall t=0,...,T ~ \text{a.s.}, \\
 & \langle S_t(\xi_{0:t}),X_t(\xi_{0:t}) -X_{t-1}(\xi_{0:t-1}) \rangle \leq G_t(\xi_{0:t})+ u_t(\xi_{0:t}), ~ \forall t=1,...,T  ~ \text{a.s.} \}.
\end{align*}
Then, the property follows from the lower semicontinuity of the indicator function.
Next, note that
    \begin{align*}
        \liminf_{\nu\to\infty} -\mathbb{E}^{p\nup} \langle S_{T+1},X_T\rangle
        = \lim_{\nu\to\infty} \sum_{j=1}^m p_j\nup \langle -S_{T+1}(j),X_T\nup(j^-)\rangle
        &= \sum_{j=1}^m p_j \langle -S_{T+1}(j),X_T(j^-)\rangle,
    \end{align*}
which allows us to conclude that $\liminf_{\nu\to\infty} f_2\nup(u\nup,X\nup)\geq  f_2(u,x)$.

Consider now $(u,X) \in \R^{N_T} \times \R^n$. Without loss of generality, assume that $X \in \mathcal{X}(u)$ so that $f_2(X,u)<\infty$. For every $\nu\in\N$ define $X\nup := X$ and $ u\nup := u$, so then
\begin{align*}
    \limsup_{\nu \to\infty} f_2\nup(u\nup,X\nup)     & = \lim_{\nu\to\infty} \sum_{j=1}^m p_j\nup \langle -S_{\scriptscriptstyle T+1}(j),X_T(j^-)\rangle
     = \sum_{j=1}^m p_j \langle -S_{\scriptscriptstyle T+1}(j),X_T(j^-)\rangle
    = f_2(u,X). ~
\end{align*}
\end{proof}

\subsection{Rockafellian perturbation of the probability and claim}
We now study the effect of perturbing both the probability measure and the feasibility set, combining thus the Rockafellians \eqref{R1.1}  and \eqref{R2} from the previous subsections. The goal is to analyze the effect on the contingent claim of introducing perturbations in the probability measure and see if the stability that both Rockafellians have is maintained. We thus consider the Rockafellian (for $\varphi$) $f_3: \R^m\times\R^{N_T} \times \R^n \rightarrow\overline{\R} $ defined as
\begin{equation}
    f_3(u,v,X) := -\sum_{j=1}^m (p_j+u_j) \langle S_{\scriptscriptstyle T+1}(j), X_T(j^-)\rangle + \iota_{\mathcal{X}(v)}(X) + \frac{1}{2} \theta \|u\|^2_2 + \iota_{[0,+\infty[^{m}} (p+u) + \iota_{]-\infty,0]^m}(u),
    \label{R3}
\end{equation}
and the approximating Rockafellian (for $\varphi^\nu$) $f_3\nup: \R^m\times\R^{N_T} \times \R^n \rightarrow\overline{\R} $ defined by
\begin{equation}
    f_3\nup(u,v,X) := -\sum_{j=1}^m (p_j\nup+u_j) \langle S_{T+1}(j), X_T(j^-)\rangle + \iota_{\mathcal{X}(v)}(X) + \frac{1}{2} \theta\nup \|u\|^2_2 + \iota_{[0,+\infty[^{m}} (p\nup+u) + \iota_{]-\infty,0]^m}(u).
    \label{AR3}
\end{equation}

The next result is analogous to Propositions \ref{prop:epi1} and \ref{prop:epi2}.

\begin{proposition}\label{prop:epi3}
    If $p\nup \rightarrow p$ and $\theta \nup \rightarrow \theta$, then the sequence of approximating Rockafellians $f_3\nup$ epi-converges to the Rockafellian $f_3$.
\end{proposition}

\begin{proof}
We show that Conditions \ref{epi:liminf} and \ref{epi:limsup} for epi-convergence hold. Let $(u,v,X)\in\R^m\times\R^{N_T}\times\R^n$ and $(u^\nu,v^\nu,X^\nu)\subset \R^m\times \R^{N_T}\times\R^n $ be such that $(u^\nu,v^\nu,X^\nu) \rightarrow (u,v,X)$. Since the sets $]-\infty,0]^m$ and $[0,+\infty[^m$ are closed, the functions  $\iota_{]-\infty,0]^m}(u)$ and $\iota_{[0,+\infty[^{m}}(p+u)$ are lsc and, by the same argument used in Proposition 3.1, it follows that $\iota_{\mathcal{X}(v)}(X)$ is lsc.

From the convergence of the sequences, we have
\begin{align*}
    \liminf_{\nu \to\infty} \sum_{j=1}^m (p_j^{\nu}+u_j^{\nu}) \langle -S_{T+1}(j), X_T^{\nu}(j^-) \rangle
    & = \sum_{j=1}^m \lim_{\nu \to\infty}  (p_j^{\nu}+u_j^\nu) \langle -S_{T+1}(j), X_T\nup(j^-) \rangle \\
    &= \sum_{j=1}^m (p_j+u_j) \langle-S_{T+1}(j), X_T(j^-)\rangle,
\end{align*}
from which we conclude that $\liminf_{\nu \to\infty} f_3^{\nu}(u^\nu,v^\nu,X^\nu) \geq f_3(u,v,X)$.

Now let $(u,v,X) \in \R^m \times \R ^{N_T} \times \R^n $. Without loss of generality, we assume that $X\in \mathcal{X}(v)$, $p+u \in [0,+\infty[^m$ and $u \leq 0$ so that $f_3(u,v,X)<\infty$. For every $\nu\in\N$, we define $X\nup :=  X$, $v\nup :=  v$ and $ u_j\nup := \min \{0,p_j+u_j-p_j\nup\}$ for every $j\in\{1,\dots,m\}$. Note that $u \nup \leq 0$, $u\nup + p\nup \geq 0$  and $ u\nup \rightarrow u$. Then
\begin{align*}
    \limsup_{\nu\to\infty} f_3\nup (u\nup,v\nup,X\nup) & = \iota_{\mathcal{X}(v)}(X) + \sum_{j=1}^m \lim_{\nu\to\infty} (p_j\nup + \min \{0, p_j+u_j-p\nup_j\}) \langle-S_{T+1}(j), X_T(j^-)\rangle \\
    & \hspace{.5cm} + \lim_{\nu\to\infty} \frac{1}{2} \theta \nup \sum_{j=1}^m (\min \{0, p_j+u_j-p\nup_j\}) ^2 \\
    & = f_3(u,v,X). \quad
\end{align*}
\end{proof}

\subsection{Failure of epi-convergence under perturbations of critical scenarios} \label{sec:counterproposition}

While probability convergence often implies solution proximity, this property can fail if catastrophic events carry vanishingly small probabilities but arbitrarily large impacts (see the example in \cite{antil2024rockafellian}). We formalize this instability below, noting that it does not invalidate our previous results since it relies on a different perturbation mechanism than Section~\ref{sec:3.1}. Recall that $\tilde{I}$ denotes the set of catastrophic nodes, while $I$ denotes the set of original nodes.

\begin{proposition}\label{prop:no epi}
    Suppose $p^\nu \to p$, and the asset prices in the critical scenarios are scaled by $(\kappa_j^i)^\nu$ such that $\liminf_{\nu \in \N} p_j^\nu (\kappa_j^i)^\nu > 0$ for some $j\in \tilde I$ some risky asset $i$. Then, in such setting with critical events, any approximating Rockafellian fails to epi-converge to a Rockafellian of the original problem.
\end{proposition}

\begin{proof}
Let $S_{T+1}\nup(j)$ denote the asset prices in scenario $j$ at time $T+1$, explicitly indicating their dependence on $\nu$. That is, for every $j\in\tilde I$, $S_{T+1}^\nu(j)= (\kappa_j^i)^\nu S_{T}(j^-)$.

To show that epi-convergence fails, we test Condition \ref{epi:liminf} at $u = 0$ for an $X \in \mathcal{X}$ where $X_T^i$ is strictly positive in at least one critical scenario. Let $X^\nu \to X$ and assume $\liminf_{\nu \in \N} f^{\nu}(0,X^\nu)<\infty$ without loss of generality. We have that
\begin{align*}
     \liminf_{\nu \in \N} f^{\nu}(0,X\nup) = & \liminf_{\nu \in \N} ~ \sum_{j=1}^m p_j^{\nu} \langle -S_{T+1}\nup(j), X_T^{\nu}(j^-) \rangle \\
   = &  \liminf_{\nu \in \N}~   \bigg(\sum_{j \in I} p_j\nup \langle -S_{T+1}(j), X_T^{\nu}(j^-)\rangle+ \sum_{j \in \tilde I} p_j\nup\langle -S_{T+1}\nup(j), X_T^{\nu}(j^-)\rangle \bigg) \\
    = & \sum_{j \in I} p_j \langle -S_{T+1}(j), X_T(j^-)\rangle+\liminf_{\nu \in \N}  \sum_{j \in \tilde I} p_j\nup \langle -S_{T+1}\nup(j), X_T^{\nu}(j^-)\rangle
\end{align*}
The first term matches $f(0,X)$ for any Rockafellian $f$. However, evaluating the second term under our growth assumption reveals that
\begin{align*}
   \liminf_{\nu \in \N} \sum_{j \in \tilde I} p_j^\nu \langle -S_{T+1}^\nu(j), X_T^\nu(j^-)\rangle \leq -\sum_{j \in \tilde I} \sum_{i=0}^k \liminf_{\nu \in \N} p_j^\nu (\kappa_j^i)^\nu S_{T}^i(j^-)\cdot (X_T^i(j^-))^\nu < 0,
\end{align*}
which prevents $\liminf_{\nu \in \N} f^{\nu}(0,X^\nu) \geq f(0,X)$ from holding. Thus $f^{\nu}\xrightarrow{e} f$  fails.
\end{proof}

This phenomenon underscores the boundaries of the stability framework under Rockafellian perturbations. As will be illustrated in Subsection~\ref{sec:counterexample}, introducing catastrophic events with vanishing probabilities but unbounded asset prices yields a sequence of problems where the probability distributions converge to the baseline model, yet the optimal solutions remain unstable. Consequently, the approximating Rockafellian functions fail to preserve stability in the sense of epi-convergence.

\section{Dual stability and shadow prices}
\label{sec:4}

We now turn to the dual problems associated with the Rockafellians introduced in Section~\ref{sec:3}. We interpret the dual variables as shadow prices, study the stability of the dual functions through hypo-convergence, and establish strong duality together with a characterization of the duality gap whenever strong duality fails.  Before proceeding, we comment on the role of the no-short-selling constraint imposed on the investment problem.

\begin{remark}\label{rem:positivity}
The no-short-selling constraint $X \geq 0$ is crucial for ensuring that the dual problem remains non-trivial. In contrast, the primal problem is well-defined under the no-arbitrage assumption.
Indeed, if this constraint is removed from \eqref{max_problem}, the feasible set loses compactness and thus admits a recession direction $d\in\R^n$ satisfying $\langle S_0, d_0 \rangle \leq 0$ and $\langle S_t, d_t - d_{t-1} \rangle \leq 0$ for every $t \in \{1, \dots, T\}$, with $\langle S_{T+1}(j), d_T(j^-) \rangle > 0$ for some scenario $j\in\{1,\dots,m\}$.
Consequently, the dual function
degenerates: for the probability-perturbation Rockafellian, choosing $\bar{u}_j = 0$ and $\bar{u}_k = -p_k$ for all $k \neq j$ yields
\begin{equation*}
\psi_1(y) \leq \sum_{k=1}^m (p_k + \bar{u}_k) \, g_k(X + \lambda d) - \langle y, \bar{u} \rangle = p_j \, g_j(X + \lambda d) - \langle y, \bar{u} \rangle \longrightarrow -\infty
\end{equation*}
as $\lambda \to \infty$, which implies $\psi_1 \equiv -\infty$. The same direction $d$ prevents the approximating tilted Rockafellian from being tight.
\end{remark}

\subsection{Probability perturbation}

\subsubsection{Lagrangian and dual problem}

We study the dual problems through the Lagrangians associated with the Rockafellians $f_1$ and $\tilde f_1$, defined in \eqref{R1} and \eqref{R1.1}, respectively. Cases 1.1 and 1.2 correspond to $f_1$ with $\theta = 0$ and $\theta > 0$, respectively, while Cases 2.1 and 2.2 correspond to $\tilde f_1$ under the same distinction. For convenience, we define the function $g_j(X):=\langle -S_{T+1}(j),X_T(j^-)\rangle$, and note that $g_j(X) \le 0$ for every $X \in \mathcal{X}$, since asset prices are nonnegative and short-selling is prohibited.

\setcounter{case}{0}
\renewcommand{\thecase}{1.\arabic{case}}

\refstepcounter{case}
\textbf{Case \thecase: Lagrangian associated with $f_1$ and $\theta=0$}. \label{case:1.1}

For any $X \in \R^n, y \in \R^m$, the Lagrangian associated with $f_1$ is given by
$$l_1(X,y) = \inf_{u \in \R^m} \{ f_1(u,X) - \langle y,u \rangle \} = \iota_{\mathcal{X}}(X) + \sum_{j=1}^m h_j(X,y),$$
with the functions
$$h_j(X,y) = \begin{cases}
    p_jy_j, \text{ if } g_j(X)  \geq y_j, \\
    -\infty, \text{ if } g_j(X) < y_j.
\end{cases}$$

Then, the dual function is given by
$$\psi_1(y) = \min_{X \in \R^n} l_1(X,y) = \begin{cases}
    -\infty, \text{ if } g_j(X) < y_j \text{ for some } X \in \mathcal{X} \text{ and } j\in\{1,..., m\}, \\
    \sum_{j=1}^m p_jy_j, \text{ if } g_j(X) \geq y_j \text{ for every } X \in \mathcal{X} \text{ and } j\in\{1,..., m\},
\end{cases}
$$
and the set of maximizers of $\psi_1$ is
\[
\argmax_{y\in\R^m} \psi_1(y) =\Pi_{j=1}^m Y_j,\quad\mbox{where}\quad
Y_j=\begin{cases}
    \{\min_{x \in \mathcal{X}} g_j(x)\}, \text{ if } p_j \neq 0, \\
    ]-\infty, \min_{x \in \mathcal{X}} g_j(x) ], \text{ if } p_j=0.
\end{cases}\]

\refstepcounter{case}
\textbf{Case \thecase: Lagrangian associated with $f_1$ and $\theta>0$}. \label{case:1.2}

For any $X \in \R^n, y \in \R^m$, the Lagrangian associated with $f_1$ is given by
$$l_1(X,y)= \sum_{j=1}^m h_j(X,y) + \iota_{\mathcal{X}}(X),$$
where
$$h_j(X,y) = \begin{cases}
    \frac{1}{2} \theta p_j^2 + y_jp_j, \text{ if } g_j(X) \in ]y_j+ \theta p_j, +\infty], \\
    p_jg_j(X) - \frac{1}{2\theta} (y_j - g_j(X))^2, \text{ if } g_j(X) \in ]-\infty, y_j + \theta p_j].
\end{cases}$$

Despite not having an explicit formulation of the dual function, we have that
$$
\sup_{y \in \mathbb{R}^m} \psi_1(y) \in [W, V],$$
where $ V := \min_{X \in \mathcal{X}} \varphi(X)$ and $W := \sum_{j=1}^m p_j \min_{X \in \mathcal{X}} g_j(X).$ This result is formally presented in Theorem~\ref{thm:SD1}.

\setcounter{case}{0}
\renewcommand{\thecase}{2.\arabic{case}}
\refstepcounter{case}
\textbf{Case \thecase: Lagrangian associated with $\tilde f_1$ and $\theta = 0$}. \label{case:2.1}

In this case, the Lagrangian takes the form
$$\tilde l_1(X,y) = \iota_{\mathcal{X}}(X) + \sum_{j=1}^m \tilde h_j(X,y),\quad\mbox{where}\quad
\tilde h_j(X,y) = \begin{cases}
    p_jy_j, \text{ if } g_j(X) \geq y_j, \\
    p_jg_j(X), \text{ if } g_j(X) < y_j,
\end{cases}
$$
and the dual function is given by
$$
\tilde \psi_1(y)=\min_{X \in \R^n} \tilde l_1(X,y) = \min_{X \in \mathcal{X}} \sum_{j=1}^m \tilde h_j(X,y) =\min_{X \in \mathcal{X}} \sum_{j=1}^m p_j \min \{g_j(X), y_j\}.
$$
Moreover, we have
\[
\argmax_{y\in\R^m} \tilde\psi_1(y) = \{ y\in\R^m:~ \tilde\psi_1(y)=\min_{X \in \mathcal{X}} \varphi(X)\}.
\]

\refstepcounter{case}
\textbf{Case \thecase: Lagrangian associated with $\tilde f_1$ and $\theta >0$}. \label{case:2.2}

In this case, the Lagrangian takes the form
$$\tilde l_1(X,y) = \iota_{\mathcal{X}}(X) + \sum_{j=1}^m \tilde h_j(X,y),$$
where
$$ \tilde h_j(X,y) = \begin{cases}
    \frac{1}{2} \theta p_j^2 + y_jp_j, \text{ if } g_j(X) \in ]y_j+ \theta p_j, +\infty], \\
    p_jg_j(X) - \frac{1}{2\theta} (y_j - g_j(X))^2, \text{ if } g_j(X) \in [y_j, y_j + \theta p_j], \\
    p_j g_j(X), \text{ if } g_j(X) \in ]-\infty, y_j[.
\end{cases}$$
Regarding the dual function, we have
\[
\argmax_{y\in\R^m} \tilde \psi_1(y) =\Pi_{j=1}^m Y_j,\quad\mbox{where}\quad
Y_j=\begin{cases}
    \R^m_+, \text{ if } p_j \neq 0, \\
   \R, \text{ if } p_j=0.
\end{cases}\]

The calculations can be found in Appendix~\ref{A}.

\begin{remark}
We point out that the same formulas apply for the dual functions, if we consider the approximating Rockafellians $f_1\nup$ and $\tilde f_1 \nup$ and replace  $p$ with $p\nup$.
\end{remark}

\subsubsection{Tightness of approximating tilted Rockafellian}

In this section, we are interested in studying the tightness of the approximating tilted Rockafellian
$$f\nup_{y\nup}(u,X):= f_1\nup(u,X) - \langle y\nup,u \rangle,$$
with $f_1\nup$ defined in~\eqref{AR1}. This is an important property which is used for establishing the hypo-convergence of the dual functions, see  \cite[Thm. 5.1]{deride2024approximations}. It turns out that tightness fails when the regularization parameters $\theta^{\nu} \equiv 0$. However, we still have hypo-convergence of the dual functions, as shown in  Theorem~\ref{thm:hypo1}.

\begin{lemma}[Tightness of the approximating tilted Rockafellian]\label{lem:tight of AR1}
Let $\{f\nup_{y\nup}\}_{\nu \in \N}$ be the approximating tilted Rockafellian associated with $f_1^\nu$.
\begin{enumerate}[label=(\roman*)]
    \item If $\liminf_{\nu \in \N} \theta\nup > 0$ and $y\nup \to y$, then $\{f\nup_{y\nup}\}_{\nu \in \N}$ is tight.
    \item Let $\liminf_{\nu\in\mathbb{N}}\theta^\nu=0$ and $y^\nu \to y$.

\begin{enumerate}
    \item[(a)] If there exist $X\in\mathcal{X}$ and $j\in\{1,\ldots,m\}$ such that $g_j(X)<y_j,$ then the sequence $\{f_{y^\nu}^\nu\}_{\nu\in\mathbb{N}}$ is not tight.

    \item[(b)] If
    $y_j\leq\min_{X\in\mathcal{X}}g_j(X)$ for every $j \in \{1,...,m\}$. Define $\mathcal{J}$ as the set of indices where equality holds and if $y_j^\nu\nearrow y_j,\forall\,j \in \mathcal{J}$, then $\{f_{y^\nu}^\nu\}_{\nu\in\mathbb{N}}$ is tight.
\end{enumerate}
\end{enumerate}
\end{lemma}

\begin{proof}
(i) Set $c := \tfrac{1}{2}\liminf_\nu \theta\nup > 0$ and choose $\nu_0 \in \N$ such that $\theta\nup \geq c$ and $\|y\nup - y\|_\infty \leq 1$ for every $\nu \geq \nu_0$.

For each $X \in \mathcal{X}$ and $\nu \geq \nu_0$, the map $u \mapsto f\nup_{y\nup}(u, X)$ is strongly convex with modulus $\theta\nup \geq c$ on its effective domain $\{u \in \R^m : p\nup + u \geq 0\}$. Its unique minimizer is
\[
\bar u_j\nup(X) = \max\!\left\{\frac{y_j\nup - g_j(X)}{\theta\nup},\ -p_j\nup\right\}, \qquad \forall j\in\{1, \dots, m\}.
\]
From the compactness of $\mathcal{X}$, we have that $G := \max_{j} \max_{X \in \mathcal{X}} |g_j(X)| < \infty$. Combining this with $|y_j\nup| \leq |y_j| + 1$, $\theta\nup \geq c$, and $p_j\nup \in [0,1]$, we have
\[
|\bar u_j\nup(X)| \leq R := \max\!\left\{\frac{G + \|y\|_\infty + 1}{c},\ 1\right\}, \quad \forall \nu \geq \nu_0,\ \forall X \in \mathcal{X},\ \forall j\in\{1, \dots, m\}.
\]
Set $B := [-R, R]^m \times \mathcal{X}$, compact as a product of compact sets. By construction, the infimum of $f\nup_{y\nup}$ over $\R^m \times \R^n$ is attained inside $B$ for every $\nu \geq \nu_0$, so for any $\varepsilon > 0$,
\[
\inf_{(u, X) \in B} f\nup_{y\nup}(u, X) = \inf_{(u, X) \in \R^m \times \R^n} f\nup_{y\nup}(u, X) \leq \inf_{(u, X) \in \R^m \times \R^n} f\nup_{y\nup}(u, X) + \varepsilon.
\]

which implies directly that $\{f\nup_{y\nup}\}_{\nu \in \N}$ is tight.

(ii)(a) Suppose that $\liminf \theta\nup =0$, $y\nup \to y$ and there exist $X \in \mathcal{X}$ such that $g_j(X) < y_j$ for at least one $j \in \{1,...,m\}$.

\begin{itemize}
    \item[$\bullet$] If $\theta^\nu = 0$ for infinitely many terms, then, by selecting the corresponding subsequence (which, for convenience, we continue to index by $\nu$), we obtain that \begin{equation}
\label{rock lem 1}
    f\nup_{y\nup}(u,X) = \sum_{j=1}^m (p_j\nup+u_j)g_j(X) - u_jy_j\nup= \sum_{j=1}^m p_j\nup g_j(X)+ u_j(g_j(X)-y_j\nup).
\end{equation}
Noting that $g_j(X) - y_j\nup < 0$ for $\nu$ large enough, taking $u_j \rightarrow \infty$, the tightness fails.
\item[$\bullet$] If $\theta^\nu > 0$ for all $\nu \geq \nu_0$, since $\liminf_{\nu \in \N} \theta^\nu = 0$,
there exists a subsequence (we continue to index by $\nu$) $\theta^{\nu} \rightarrow 0$. Using the results obtained
in (i) for the optimal value in $u$, we have that
$$\bar u_j^{\nu}(X) = \max \left\{ \frac{y_j^{\nu} - g_j(X)}{\theta^{\nu}},\ -p_j^{\nu} \right\} \rightarrow \infty$$

and, as before, the tightness fails.
\end{itemize}

(ii)(b) Let $y_j \leq \min_{X\in\mathcal{X}} g_j(X)$ for all $j$, and $\{y^\nu\}_{\nu\in\mathbb{N}}\subset\mathbb{R}^m$ is a sequence such that $y_j^\nu \nearrow y_j$ for every $j \in \mathcal{J}$, then
\[
\bar u_j^\nu(X)=
\begin{cases}
-p_j^\nu, & \text{if } \theta^\nu=0,\\[2mm]
\displaystyle
\max\!\left\{
\frac{y_j^\nu-g_j(X)}{\theta^\nu},
\,-p_j^\nu
\right\}, & \text{if } \theta^\nu>0.
\end{cases}
\]
Since $y_j^\nu\le y_j\le \min_{X\in\mathcal{X}} g_j(X)$, it follows that
\[
\bar u_j^\nu(X)\in[-p_j^\nu,0]\subset[-1,0]
\]
in both cases. Therefore, taking $B=\mathcal{X}\times[-1,0],$ the sequence $\{f_{y^\nu}^\nu\}$ is tight.
\end{proof}

\begin{remark}\label{rem:tightness-tilde}
    For the approximating Rockafellian $\tilde f_1\nup$, defined in~\eqref{AR1.1}, the sequence $\{\tilde f\nup_{y\nup}\}_{\nu \in \N}$ is also tight for any sequences $\{y\nup\} \subset \R^m$ and $\{\theta^\nu\}\subset\R_+$. Indeed, the same argument used for $\mathcal{X}$ in Lemma~\ref{lem:tight of AR1} applies. In the perturbation space, we have
    $-p\nup_j \leq u_j \leq 0$.  Thus, since $p^\nu \in \Delta_{m-1}$, we define $U := [-1,0]^m$ and set $B := U \times \mathcal{X}.$ Then, for  all $\nu \in\N$ and $\varepsilon>0$, we obtain
$$\inf_{(u,X) \in B} f_{y\nup}\nup(u,X) \leq \inf_{(u,X) \in \R^m \times \R^n} f\nup_{y\nup}(u,X) + \varepsilon.$$
\end{remark}

\subsubsection{Hypo-convergence of dual functions}

\begin{theorem}[Hypo-convergence of Dual Functions]\label{thm:hypo1}
Let $\{p\nup\}_{\nu \in \N}$ be a sequence of probabilities converging to $p$ and let $\{\theta\nup\}_{\nu \in \N}$ be a sequence of regularizing parameters converging to $\theta$.
\begin{enumerate}[label=(\roman*)]
\item If $\theta = 0$, then the approximating dual function $\psi_1^\nu$ associated with $f_1^\nu$ hypo-converges to the dual function $\psi_1$ associated with $f_1$ (as in Case~\ref{case:1.1}).
\item If $\theta > 0$, then the approximating dual function $\psi_1^\nu$ associated with $f_1^\nu$ hypo-converges to the dual function $\psi_1$ associated with $f_1$ (as in Case~\ref{case:1.2}).
\item The approximating dual function $\tilde\psi_1^\nu$ associated with $\tilde f_1^\nu$ hypo-converges to the dual function $\tilde\psi_1$ associated with $\tilde f_1$ (as in Cases~\ref{case:2.1} and \ref{case:2.2}).
\end{enumerate}
\end{theorem}

\begin{proof}
(i) We note that $\operatorname{dom}(-\psi) = \left\{y \in \mathbb{R}^m : y_j \leq \min_{X \in \mathcal{X}} g_j(X) \ \forall j \in \{1,...,m\} \right\}$.
From Proposition~\ref{prop:epi1}, the epi-convergence $f_1^\nu \xrightarrow{e} f_1$ follows.
Furthermore, considering any sequence $\{y^\nu\}_{\nu \in \mathbb{N}} \subset \mathbb{R}^m$ such that
$y_j^\nu \nearrow y_j$ with $y \in \operatorname{dom}(-\psi)$, we have that $\{f^\nu_{y^\nu}\}$
is tight by Lemma~\ref{lem:tight of AR1}. Finally, hypo-convergence of the dual functions
then follows from \cite[Theorem~5.1(a)]{deride2024approximations}.

(ii)-(iii) From Proposition~\ref{prop:epi1}, the epi-convergence $f_1^\nu \xrightarrow{e} f_1$ follows. Moreover, tightness of $\{f^\nu_{y^\nu}\}$ holds in Case~\ref{case:1.2} by Lemma~\ref{lem:tight of AR1} (since $\liminf_\nu \theta^\nu > 0$) and in Cases~\ref{case:2.1} and \ref{case:2.2} by Remark~\ref{rem:tightness-tilde}. Hypo-convergence of the dual functions then follows from \cite[Theorem~5.1(a)]{deride2024approximations}.
\end{proof}

\subsubsection{Strong duality}

\begin{theorem}[Strong Duality]\label{thm:SD1}
Let $V := \min_{X \in \mathcal{X}} \varphi(X)$ denote the primal value, and define the \emph{wait-and-see value}
\[
W := \sum_{j=1}^m p_j \min_{X \in \mathcal{X}} g_j(X).
\]
\begin{enumerate}[label=(\roman*)]
\item In Cases~\ref{case:2.1} and \ref{case:2.2}, the dual value equals $V$. Strong duality holds.
\item In Case~\ref{case:1.1}, the dual value equals $W$, and the duality gap $V - W$ is non-negative. The gap vanishes if and only if there exists $X^* \in \mathcal{X}$ such that
\[
g_j(X^*) = \min_{X \in \mathcal{X}} g_j(X) \quad \text{for every } j \text{ with } p_j > 0.
\]
\item In Case~\ref{case:1.2}, the dual value $\sup_y \psi_1(y)$ lies in $[W, V]$. Strong duality holds whenever a common minimizer $X^*$ as in (ii) exists; in particular, the regularization term does not introduce a new duality gap beyond the wait-and-see gap of Case~\ref{case:1.1}.
\end{enumerate}
\end{theorem}

\begin{proof}
(i) By the calculations of Cases~\ref{case:2.1} and \ref{case:2.2}, $0 \in \argmax \tilde\psi_1$. Evaluating $\tilde\psi_1$ at $y = 0$ and using $g_j(X) \le 0$ on $\mathcal{X}$, both case formulas reduce to $\tilde\psi_1(0) = \min_{X \in \mathcal{X}} \sum_j p_j g_j(X) = V$. Hence $\sup_y \tilde\psi_1(y) = V$.

(ii) The argmax of $\psi_1$ in Case~\ref{case:1.1} consists of vectors $y$ with $y_j = \min_{X \in \mathcal{X}} g_j(X)$ for every $j$ with $p_j > 0$, whence
\[
\sup_y \psi_1(y) = \sum_{j=1}^m p_j \min_{X \in \mathcal{X}} g_j(X) = W.
\]
The inequality $V \ge W$ follows from $\sum_j p_j g_j(X) \ge \sum_j p_j \min_{X'} g_j(X')$ for every $X \in \mathcal{X}$, by taking the infimum on the left.

Suppose first that some $X^* \in \mathcal{X}$ satisfies $g_j(X^*) = \min_{X} g_j(X)$ for every $j$ with $p_j > 0$. Then $V \le \sum_j p_j g_j(X^*) = W$, which combined with $V \ge W$ gives $V = W$. Conversely, if $V = W$ and $X^* \in \argmin\varphi$, then $0 = V - W = \sum_{j: p_j > 0} p_j(g_j(X^*) - \min_X g_j(X))$, and since each summand is non-negative with $p_j > 0$, we conclude $g_j(X^*) = \min_X g_j(X)$ for every $j$ with $p_j > 0$.

(iii) The inequality $\sup_y \psi_1(y) \le V$ is weak duality. For the lower bound, the inequality $\psi_1(y) \ge L(y) := \sum_j \min_X h_j(X, y)$ (pointwise minimization) combined with the maximum-over-$y$ calculation and defining $m_j$ as $m_j :=\min_X g_j(X)$, we obtain $\sup_y L(y) = \sum_j p_j m_j = W$ gives $\sup_y \psi_1(y) \ge W$.

Suppose now that a common minimizer $X^*$ exists, that is, $g_j(X^*) = m_j$ for every $j$ with $p_j > 0$. Choose $y = m$ (so $y_j = m_j$). Then at $X^*$, the lower-region branch of $h_j$ gives $h_j(X^*, m) = p_j m_j$. For any other $X \in \mathcal{X}$, $g_j(X) \ge m_j$, and since $h_j(\cdot, m)$ is non-decreasing in $g_j$, $h_j(X, m) \ge p_j m_j$. Hence $\psi_1(m) = \sum_j p_j m_j = W = V$, establishing strong duality.
\end{proof}

\begin{remark}[Strict duality gap in Case~\ref{case:1.1}]\label{rem:strict-gap}
The duality gap in Case~\ref{case:1.1} is strictly positive whenever the family $\{g_j : p_j > 0\}$ does not admit a common minimizer in $\mathcal{X}$. This is the generic situation: if, at some parent node $i \in I_T$, two children with positive probability favor different optimal allocations of $X_T(i)$, then no common minimizer exists. The gap vanishes only in degenerate cases: for example, when, at every parent node, the price vectors $S_{T+1}(j)$ corresponding to all children with positive probability induce the same optimal allocation.
\end{remark}

\begin{lemma}[Vanishing of the duality gap as $\theta \to \infty$]\label{lem:theta-limit}
Let $D(\theta) := \sup_{y \in \R^m} \psi_1(y)$ denote the Case~\ref{case:1.2} dual value as a function of the regularization parameter $\theta > 0$. Then $D(\theta) \to V$ as $\theta \to \infty$.
\end{lemma}

\begin{proof}
Since $\mathcal{X}$ is compact and each $g_j$ is continuous, $G := \max_j \max_{X \in \mathcal{X}} |g_j(X)| < \infty$. Choose $\theta_0$ large enough that $\theta_0 p_j \geq G$ for every $j$ with $p_j > 0$; then for all $\theta \geq \theta_0$ and every $X \in \mathcal{X}$, $g_j(X) \leq G \leq \theta p_j \leq y_j + \theta p_j$ at $y=0$, so evaluating $h_j(\cdot, 0)$ falls in the lower-region branch for every $j$ and every $X \in \mathcal{X}$. Hence
\[
\psi_1(0) = \min_{X \in \mathcal{X}} \sum_{j=1}^m \left( p_j g_j(X) - \frac{g_j(X)^2}{2\theta} \right) = \min_{X \in \mathcal{X}} \left( \varphi(X) - \sum_{j=1}^m \frac{g_j(X)^2}{2\theta} \right).
\]
The correction term is bounded by $mG^2/(2\theta)$ uniformly over $X \in \mathcal{X}$, so
\[
V - \frac{mG^2}{2\theta} \leq \psi_1(0) \leq V,
\]
for all $\theta \geq \theta_0$. Since $D(\theta) \geq \psi_1(0)$ (evaluating the sup at $y=0$) and $D(\theta) \leq V$ (weak duality), we conclude $V - mG^2/(2\theta) \leq D(\theta) \leq V$, and letting $\theta \to \infty$ gives $D(\theta) \to V$.
\end{proof}

\begin{remark}[Regularization and the duality gap in Case~\ref{case:1.2}]\label{rem:theta-gap}
Theorem~\ref{thm:SD1}(iii) shows that the regularization parameter $\theta$ does not by itself create a duality gap: when a common minimizer exists, strong duality holds for every $\theta \ge 0$. In the absence of a common minimizer, Lemma~\ref{lem:theta-limit} shows that the gap closes as $\theta \to \infty$, i.e., $D(\theta) \to V$. We conjecture that $D(\theta)$ is in addition strictly increasing in $\theta$ for $\theta>0$, so that the gap $V-D(\theta)$ decreases monotonically from $V-W$ toward $0$ without reaching either endpoint at finite $\theta$. We do not have a proof of strict monotonicity, nor of a uniform lower bound on $V - D(\theta)$ away from $0$ for finite $\theta$; a complete characterization of $D(\theta)$ is left for future work.
\end{remark}

\subsubsection{Shadow prices} \label{sec:4.1.3}

In this section, we analyze the economic interpretation of the dual variables viewed as shadow prices associated with the constraints. Recall that the primal problem is in minimization form, so dual variables expressing prices appear with the opposite sign.

\emph{Case~\ref{case:1.1}.} Let $y^*$ be an optimal solution of the dual problem. By Theorem~\ref{thm:SD1}(ii),
\[
-\sup_{y \in \R^m} \psi_1(y) = -W = \sum_{j=1}^m p_j \max_{X \in \mathcal{X}} \langle S_{T+1}(j), X_T(j^-)\rangle,
\]
so $-W$ is the expected value of the wait-and-see strategy: in each scenario $j$, the maximum portfolio value attainable when the realization is observed before the allocation. Theorem~\ref{thm:SD1}(ii) thus identifies the duality gap as the \emph{expected value of perfect information}: the cost of being unable to anticipate future realizations. By Remark~\ref{rem:strict-gap}, this gap is strictly positive in generic markets.

For each $j$ with $p_j > 0$, the component $-y_j^* = \max_{X \in \mathcal{X}} \langle S_{T+1}(j), X_T(j^-)\rangle$ is the maximum portfolio return attainable in scenario $j$, the value of having perfect information conditional on scenario $j$. For $j$ with $p_j = 0$, the variable $y_j$ does not contribute to the dual value and admits any value in $(-\infty, \min_{X\in\mathcal{X}} g_j(X)]$.

\emph{Case~\ref{case:1.2}.} The regularization term $\frac{1}{2}\theta\|u\|^2_2$ penalizes deviations of the probability vector from $p$. By Remark~\ref{rem:theta-gap}, the regularization does not introduce a new duality gap; rather, it partially closes the wait-and-see gap of Case~\ref{case:1.1}. Economically, the penalty makes the dual problem less responsive to extreme reallocations of probability mass, dampening the wait-and-see anticipation effect.

\emph{Case~\ref{case:2.1}.} In this case, after some simple computations, it can be proved that
\begin{equation*}
\left\{y \in \mathbb{R}^m :y_j \geq 0,~ \forall j \in J_+\right\}
\subset
\argmax_{y\in\R^m} \tilde{\psi_1}(y)
\subset
\left\{y \in \mathbb{R}^m : y_j \geq \max_{X \in \argmin \varphi}
g_j(X),~ \forall j \in J_+ \right\}.
\end{equation*}
This shows that the dual is always solvable at $y^* \ge 0$. Moreover, due to Theorem~\ref{thm:SD1}(i), the dual admits solutions with negative components. For such components, the value $-y_j^* > 0$ represents the maximum price the investor is willing to pay for additional information in scenario $j$.

\emph{Case~\ref{case:2.2}.} Here $\argmax \tilde\psi = \mathbb{R}^m_+$. The regularization term, combined with the constraint $u \le 0$, removes the incentive to acquire additional information: every dual optimum satisfies $y^* \ge 0$, and the implicit information price is zero.

By Theorem~\ref{thm:SD1}(i), Cases~\ref{case:2.1} and \ref{case:2.2} exhibit no duality gap. The constraint $u \le 0$, which forbids redistributing probability mass to scenarios with higher expected payoff, eliminates the wait-and-see channel from the dual.

\subsection{Claim perturbation}

\subsubsection{Dual function}
\label{sec:4.2.1}

For the claim perturbation, the dual function can be computed directly from a linear programming reformulation of~\eqref{P}, without resorting to a Lagrangian.

Let $I_T := \{j^- : j = 1, \dots, m\}$ denote the set of parent nodes of the terminal scenarios, and recursively $I_t := \{i^- : i \in I_{t+1}\}$ for $t = T-1, \dots, 0$, so $|I_0| = 1$ (the root) and $|I_t| = n_t$. We label the nodes of the event tree level by level from $I_0$ to $I_T$, and within each level from top to bottom, indexing the strategy $X$ blockwise by $\bigcup_{t=0}^T I_t$.

In this notation,~\eqref{P} reads as the linear program
\[
\min_{X \in \R^n} c^T X \quad \text{subject to} \quad AX \leq b, \quad X \geq 0,
\]
where $c \in \R^n$, $b \in \R^{N_T}$ (with $N_T := 1 + \sum_{t=1}^T n_t$), and $A \in \R^{N_T \times n}$ are given blockwise by
\[
c_i = \begin{cases}
    0, & i \in I_t,\ t \in \{0, 1, \dots, T-1\}, \\
    -\sum_{j \in i^+} p_j S_{T+1}(j), & i \in I_T,
\end{cases}
\qquad
b = \begin{pmatrix} G_0 \\ (G_1(i))_{i \in I_1} \\ \vdots \\ (G_T(i))_{i \in I_T} \end{pmatrix},
\]
\[
A_{k, j} = \begin{cases}
    S_0^T(1), & k = j = 1, \\
    S_t^T(k), & k \in I_t,\ j = k, \\
    -S_t^T(k), & k \in I_t,\ j = k^-, \\
    0, & \text{otherwise}.
\end{cases}
\]
Under this construction, $c^T X = \sum_{j=1}^m p_j \langle -S_{T+1}(j), X_T(j^-)\rangle = \varphi(X)$ and $\{X \geq 0 : AX \leq b\} = \mathcal{X}$.

The Rockafellian~\eqref{R2} perturbs the right-hand side of the constraints:
\[
f_2(u, X) = c^T X + \iota_{\{X \geq 0,\ AX \leq b + u\}}(X).
\]
The associated dual function follows by direct computation:
\begin{align*}
\psi_2(y) = \inf_{u, X}\{f_2(u, X) - \langle y, u\rangle\}
&= \inf_{X \geq 0} \Bigl( c^T X - \sup_{u \geq AX - b} \langle y, u\rangle \Bigr) \\
&= y^T b - \iota_{\{A^T y \leq c\}}(y) - \iota_{\{y \leq 0\}}(y),
\end{align*}
where the second equality uses $\sup_{u \geq AX - b} \langle y, u\rangle = \langle y, AX - b\rangle$ for $y \leq 0$ ($+\infty$ otherwise), and the inner minimization over $X \geq 0$ is bounded below iff $c - A^T y \geq 0$.

\begin{remark}[Dual function of the approximating Rockafellian]\label{rem:psi-nu-R2}
The dual function of~\eqref{AR2} is
\[
\psi_2^\nu(y) = y^T b - \iota_{\{A^T y \leq c^\nu\}}(y) - \iota_{\{y \leq 0\}}(y),
\]
where $c^\nu \in \R^n$ is given blockwise by $c_i^\nu = 0$ for $i \in I_t$ with $t < T$, and $c_i^\nu = -\sum_{j \in i^+} p_j^\nu S_{T+1}(j)$ for $i \in I_T$.
\end{remark}

\subsubsection{Tightness of the Approximating Rockafellian}

In this section, we are interested in studying the tightness of the approximating tilted Rockafellian
$$f\nup_{y\nup}(u,X):= f_2\nup(u,X) - \langle y\nup,u \rangle,$$
with $f_2\nup$ defined in~\eqref{AR2}.

\begin{lemma}[Tightness of $\{f^\nu_{y^\nu}\}$]\label{lem:tightness of AR2}
Let $\{y\nup\} \subset \R^{N_T}$ be a sequence satisfying $y\nup \leq 0$ and $A^T y\nup \leq c\nup$ for all $\nu \geq \nu_0$, and let $\{f\nup_{y\nup}\}_{\nu \in \N}$ denote the tilted approximating Rockafellian associated with~\eqref{AR2}. Then $\{f\nup_{y\nup}\}_{\nu \in \N}$ is tight.
\end{lemma}

\begin{proof}
At $(u, X) = (-b, 0)$, the constraints of~\eqref{AR2} read $X = 0 \geq 0$ and $AX - b = -b \leq u = -b$, so both are satisfied and $f\nup_2(-b, 0) = 0$. Hence
\[
f\nup_{y\nup}(-b, 0) = -\langle y\nup, -b\rangle = (y\nup)^T b.
\]
By Remark~\ref{rem:psi-nu-R2} and dual feasibility of $y\nup$,
\[
\inf_{(u, X) \in \R^{N_T} \times \R^n} f\nup_{y\nup}(u, X) = \psi_2\nup(y\nup) = (y\nup)^T b.
\]
Setting $B := \{(-b, 0)\}$ (compact, a singleton, independent of $\varepsilon$), we obtain for every $\varepsilon > 0$ and every $\nu \geq \nu_0$,
\[
\inf_{(u, X) \in B} f\nup_{y\nup}(u, X) = \inf_{(u, X) \in \R^{N_T} \times \R^n} f\nup_{y\nup}(u, X) \leq \inf_{(u, X) \in \R^{N_T} \times \R^n} f\nup_{y\nup}(u, X) + \varepsilon.
\]
\end{proof}

\subsubsection{Hypo-convergence of dual functions}

\begin{lemma}[Existence of an interior dual direction]\label{lem:interior-d}
There exists $d \in \R^{N_T}$ with $d \leq 0$ and $A^T d < 0$ componentwise.
\end{lemma}

\begin{proof}
We construct $d$ recursively, backward in $t$. For each $i \in I_T$, set $d_i := -1$, so the block of $A^T d$ at $i$ reads $(A^T d)_i = S_T(i) d_i = -S_T(i) < 0$ (using $S_T(i) > 0$ componentwise). Inductively, suppose $d_j < 0$ has been chosen for every $j \in I_{t+1}$. For each $i \in I_t$,
\[
(A^T d)_i = S_t(i) d_i - \sum_{j \in i^+} S_{t+1}(j) d_j,
\]
and the second term is strictly negative (componentwise) by the induction hypothesis. Choosing $d_i$ sufficiently negative ensures $(A^T d)_i < 0$ componentwise. The resulting vector satisfies $d \leq 0$ and $A^T d < 0$.
\end{proof}

\begin{theorem}[Hypo-convergence of Dual Functions]\label{thm:hypo2}
If $p\nup \to p$, then the dual function $\psi_2\nup$ of~\eqref{AR2} hypo-converges to the dual function $\psi_2$ of~\eqref{R2}.
\end{theorem}

\begin{proof}
Epi-convergence $f_2\nup \xrightarrow{e} f_2$ holds by Proposition~\ref{prop:epi2}. By Theorem~5.1(a) in~\cite{deride2024approximations}, it suffices to exhibit, for every $y \in \mathrm{dom}(-\psi_2)$, a sequence $y\nup \to y$ for which $\{f\nup_{y\nup}\}$ is tight.

Fix $y \in \mathrm{dom}(-\psi_2)$, so $y \leq 0$ and $A^T y \leq c$. Let $d$ be the vector from Lemma~\ref{lem:interior-d}, and set

\[
\varepsilon_\nu := \max_l \max\!\left\{0,\ \frac{c_l\nup - c_l}{(A^T d)_l}\right\}, \qquad y\nup := y + \varepsilon_\nu d.
\]
Both numerator and denominator in the ratio are negative when the term is non-zero, so $\varepsilon_\nu \geq 0$. By construction $A^T y\nup \leq c\nup$, and since $y \leq 0$, $d \leq 0$, $\varepsilon_\nu \geq 0$, also $y\nup \leq 0$. As $c\nup \to c$ and $A^T y \leq c$, the ratio inside the max tends to $\max\{0, c_l\nup-c_l)/(A^T d)_l\} = 0$, so $\varepsilon_\nu \to 0$ and $y\nup \to y$. Tightness of $\{f\nup_{y\nup}\}$ follows from Lemma~\ref{lem:tightness of AR2}.
\end{proof}

\subsubsection{Strong duality}

\begin{theorem}[Strong Duality for \eqref{R2}]\label{thm:SD2}
The dual value of~\eqref{R2} coincides with the primal value: $\sup_y \psi_2(y) = \min_X \varphi(X)$.
\end{theorem}

\begin{proof}
We follow Theorem~3 in~\cite{rockafellar1976stochastic}, which yields strong duality once we show that the optimal value function $\phi(u) := \inf_{X \in \R^n} f_2(u, X)$ is convex, proper, and lower semicontinuous on $\R^{N_T}$.

\emph{Properness.} The Rockafellian $f_2(u, X) = c^T X + \iota_{\{X \geq 0,\ AX \leq b + u\}}(X)$ never takes the value $-\infty$, so neither does $\phi$. Moreover, $\phi(0) = \min_{X \in \mathcal{X}} \varphi(X)$ is finite by feasibility of~\eqref{P}.

\emph{Convexity.} $f_2$ is jointly convex in $(u, X)$ (linear part plus indicator of a convex set), and $\phi$ is the inf-projection of a jointly convex function onto $u$, hence convex.

\emph{Lower semicontinuity.} For each $u \in \R^{N_T}$ with $\mathcal{X}(u) := \{X \geq 0 : AX \leq b + u\} \neq \emptyset$, the same recursive argument as in Remark~\ref{rem:compactness-X} (with $b$ replaced by $b + u$) shows that $\mathcal{X}(u)$ is compact; otherwise $\phi(u) = +\infty$. Moreover, on any compact set $K \subset \R^{N_T}$, the union $\bigcup_{u \in K} \mathcal{X}(u)$ is bounded uniformly. Consequently, the inf-projection $\phi$ is lower semicontinuous (\cite[Theorem~1.17]{rockafellar1998variational}).

Since $\phi$ is convex, proper, and lsc, $\phi^{**} = \phi$, and in particular $\sup_y \psi_2(y) = \phi^{**}(0) = \phi(0) = \min_X \varphi(X)$.
\end{proof}

\begin{proposition}\label{prop:dual tight}
    Suppose that $p\nup \rightarrow p$. Then the approximating dual function $-\psi_2^\nu(y) = -y^T b +\iota_{\{A^T y \leq c^\nu\}}(y)+ \iota_{\{y \leq 0\}}(y)$ is tight
\end{proposition}
\begin{proof}
Let $S(c)=\arg\max \psi_2(y)$, where we regard
$S:\mathbb{R}^{N_T}\rightrightarrows \mathbb{R}^{N_T}$ as a set-valued mapping and note that $S(c)$ is nonempty by the Strong Duality Theorem~\ref{thm:SD2}. Let $\bar y \in S(c)$ a solution, then by Theorem 2.4 \cite{mangasarian1987lipschitz}, there exist a solution $y\nup \in S(c\nup)$ such that $\|\bar y - y\nup\|_{\infty} \leq L \|c-c\nup\|.$ Note that, since $p^\nu \rightarrow p$ it follows that $c^\nu \rightarrow c$. Finally, taking $B_\varepsilon = B[\bar y,1]$, there exist $\nu_0 \in \N$ such that $S(c^\nu)\cap B_\varepsilon \neq\emptyset$ for all $\nu \geq \nu_0$. Hence, for every $\varepsilon>0$, $\inf_{y\in B_\varepsilon} -\psi_2^\nu(y) \leq \inf_{y\in \R^m} -\psi_2(y)\nup +\varepsilon,\quad \forall \nu \geq \nu_0$
\end{proof}

\begin{remark}\label{rem:convergencia de soluciones}
    Let $\{y\nup, \nu \in \N\}$ be the solutions of the approximating dual problems, then, by Proposition~\ref{prop:dual tight} and Theorem~\ref{thm:hypo2}, every cluster point of $\{y\nup, \nu \in \N\}$ are maximizers of $\psi_2$ and $\sup \psi\nup_2 \rightarrow \sup \psi_2$.
\end{remark}

\subsubsection{Shadow prices}
\label{sec:4.2.5}

The dual variables $y^* = (y_0^*, \{y_t(i)^*\}_{i \in I_t,\ t \geq 1})$ are Lagrange multipliers associated with the cash-flow constraints. Since the primal is in minimization form and the constraints are of $\le$ type, $y^* \leq 0$.

For each non-terminal node $i \in I_t$, the quantity $\lambda_t(i) := -y_t(i)^* \geq 0$ is the \emph{shadow price of cash at node $i$}: the marginal increase in $-\min\varphi = \max \mathbb{E}\langle S_{T+1}, X_T\rangle$ produced by relaxing the cash-flow constraint at $i$ by one unit. Strong duality (Theorem~\ref{thm:SD2}) then gives the decomposition
\[
\max_{X \in \mathcal{X}} \mathbb{E}\langle S_{T+1}, X_T\rangle = -\min_X \varphi(X) = -\sup_y \psi_2(y) = \sum_{t=0}^T \sum_{i \in I_t} \lambda_t(i)\, G_t(i),
\]
expressing the optimal portfolio value as an inner product of node-level shadow prices and the contingent claim's cash flows.

Dual feasibility $A^T y^* \leq c$ admits a clean economic reading in terms of the deflators $\lambda$. Translating the block-row constraints gives
\begin{align}
S_t(i)\, \lambda_t(i) &\geq \sum_{j \in i^+} S_{t+1}(j)\, \lambda_{t+1}(j), && i \in I_t,\ t \in \{0, \dots, T-1\}, \label{eq:sm-interior} \\
S_T(i)\, \lambda_T(i) &\geq \sum_{j \in i^+} p_j\, S_{T+1}(j), && i \in I_T. \label{eq:sm-terminal}
\end{align}
The interior condition~\eqref{eq:sm-interior} is a \emph{supermartingale property} (up to normalization) of the deflated price process: the deflated price at a parent node dominates the sum of deflated prices at its children. The terminal condition~\eqref{eq:sm-terminal} closes the recursion against the probability-weighted terminal asset prices. The asymmetry between~\eqref{eq:sm-interior} and~\eqref{eq:sm-terminal} — the absence of probabilities at non-terminal levels — is a direct consequence of the inequality (rather than equality) form of the budget constraints, which allows the investor to leave cash unused on the table.

This supermartingale structure is the discrete-time skeleton of the dual representation in $L^p$ developed in Section~\ref{sec:5}, where it becomes a property of measures absolutely continuous with respect to $\mathbb{P}$.

\subsection{Joint perturbation}

In the case of the Rockafellian defined in \eqref{R3}, the same observation as in Remark~\ref{rem:positivity} holds, arising from the combination of the two previously defined Rockafellians.

\subsubsection{Dual function}
In this section, we show that the dual function can be computed directly, without explicitly using the Lagrangian formulation.

Recall $A\in \R^{(N_T) \times n}$ and $b\in \R^{N_T}$ from Subsection~\ref{sec:4.2.1}. We define the linear cost vector $c(u)\in \R^n$ by
\begin{equation*}
    c_i(u) = \begin{cases}
        0, & \forall i \in I_t, ~\forall t \in\{1,...,T-1\}, \\
        -\sum_{j \in i^+} (p_j+u_j)S_{T+1}(j), &\forall i \in I_T,
    \end{cases}
\end{equation*}
so that $c(u)^TX = -\sum_{j=1}^m (p_j+u_j) \langle S_{T+1}(j), X_T(j^-)\rangle.$ Then we compute the dual function of \eqref{R3} in the following two cases.

\textbf{Case 1: $\theta=0$}
\begin{align}
    \psi_3(y,z) &= \inf_{u,v,X} f_3(u,v,X)- \langle y,u \rangle - \langle z,v\rangle \notag\\
    &= \inf_{\begin{array}{c}
         X \geq 0  \\
         -p \leq u \leq 0 \\
         AX \leq b+v
    \end{array}} c(u)^TX -\langle y,u \rangle-\langle z,v\rangle \notag\\
    &= \inf_{X\geq0,-p \leq u \leq 0 } \left (c(u)^TX -\langle y,u \rangle - \sup_{v \geq AX-b} \langle z,v\rangle \right ) \notag\\
    &= \begin{cases}
        \inf_{X\geq 0,-p \leq u \leq 0 } c(u)^TX -z^TAX +z^Tb - \langle y,u \rangle, \text{ if } z \leq 0  \\
        -\infty \  \text{ in other case}
    \end{cases} \notag\\
    &=  \begin{cases}
        \inf_{X\geq 0,-p \leq u \leq 0 } (c(u)-A^Tz)^TX +z^Tb - \langle y,u \rangle, \text{ if } z \leq 0  \label{*} \\
        -\infty \  \text{ in other case}
    \end{cases} \\
     &=  \begin{cases}
        \inf_{-p \leq u \leq 0 } z^Tb - \langle y,u \rangle, \text{ if } z \leq 0 \text{ and } A^Tz \leq c(u) \ \forall u \text{ such that } -p \leq u \leq 0  \\
        -\infty \  \text{ in other case}
    \end{cases}\notag
\end{align}

Note that $c(u)$ is decreasing in the sense that $c(0) \leq c(u)$ for all $u$ such that $ -p \leq u \leq 0,$ where the inequality is understood componentwise. That means the condition $A^Tz \leq c(u) \ \text{for all } u \text{ such that } -p \leq u \leq 0$ is equivalent to $A^Tz \leq c(0)$. With this, we obtain

\begin{align*}
         &=  \begin{cases}
        \inf_{-p \leq u \leq 0 } z^Tb - \langle y,u \rangle, \text{ if } z \leq 0 \text{ and } A^Tz \leq c(u) \ \forall u \text{ such that } -p \leq u \leq 0  \\
        -\infty \  \text{ in other case}
    \end{cases} \\
    &=  \begin{cases}
        \inf_{-p \leq u \leq 0 } z^Tb - \langle y,u \rangle, \text{ if } z \leq 0 \text{ and } A^Tz \leq c(0)\\
        -\infty \  \text{ in other case}
    \end{cases} \\
    &= \begin{cases}
        z^Tb+ \sum_{j \in J} p_jy_j, \text{ if } z \leq 0 \text{ and } A^Tz \leq c(0)\\
        -\infty \  \text{ in other case}
        \end{cases}
\end{align*}
Where $J=\{j : y_j \leq 0\}$.
In this case
\begin{equation}
\label{argmax psi3.1}
    \argmax \psi_3(y,z)= \Pi_{j=1}^m Y_j \times \argmax \{  z^Tb, \text{ s.t. } A^Tz \leq c(0), z \leq 0\},
\end{equation}
with $Y_j$ is equals to $\R^m_+$ if $p_j \neq 0$, and equals to $\R$ if $p_j=0$.  Note that $z$ is a solution to the dual problem of right-side Rockafellian.

\textbf{Case 2: $\theta >0$}

\begin{align}
    \psi_3(y,z) &= \inf_{u,v,X} f_3(u,v,X)- \langle y,u \rangle - \langle z,v\rangle \notag\\
    &= \inf_{\begin{array}{c}
         X \geq 0  \\
         -p \leq u \leq 0 \\
         AX \leq b+v
    \end{array}} c(u)^TX -\langle y,u \rangle-\langle z,v\rangle +\frac{1}{2}\theta\|u\|^2 \notag\\
    &= \inf_{X\geq0,-p \leq u \leq 0 } \left (c(u)^TX -\langle y,u \rangle + \frac{1}{2}\theta\|u\|^2 - \sup_{v \geq AX-b} \langle z,v\rangle \right ) \notag\\
    &= \begin{cases}
        \inf_{X\geq 0,-p \leq u \leq 0 } c(u)^TX -z^TAX +z^Tb - \langle y,u \rangle+\frac{1}{2}\theta\|u\|^2, \text{ if } z \leq 0  \\
        -\infty \  \text{ in other case}
    \end{cases} \notag\\
    &=  \begin{cases}
        \inf_{X\geq 0,-p \leq u \leq 0 } (c(u)-A^Tz)^TX +z^Tb - \langle y,u \rangle+\frac{1}{2}\theta\|u\|^2, \text{ if } z \leq 0  \label{**}  \\
        -\infty \  \text{ in other case}
    \end{cases} \\
     &=  \begin{cases}
        \inf_{-p \leq u \leq 0 } z^Tb - \langle y,u \rangle+\frac{1}{2}\theta\|u\|^2, \text{ if } z \leq 0 \text{ and } A^Tz \leq c(0) \\
        -\infty \  \text{ in other case}
    \end{cases} \notag\\
    &=\begin{cases}
        z^Tb+\inf_{-p \leq u \leq 0 }\frac{1}{2}\theta\|u\|^2  - \langle y,u \rangle, \text{ if } z \leq 0 \text{ and } A^Tz \leq c(0) \\
        -\infty \  \text{ in other case}
    \end{cases} \notag \\
     &=\begin{cases}
        z^Tb+ \sum_{j=1}^m\inf_{-p_j \leq u_j \leq 0 }\frac{1}{2}\theta u_j^2  - y_ju_j, \text{ if } z \leq 0 \text{ and } A^Tz \leq c(0) \\
        -\infty \  \text{ in other case}
    \end{cases} \notag \\
    &=\begin{cases}
        z^Tb+ \sum_{j=1}^m h_j(y), \text{ if } z \leq 0 \text{ and } A^Tz \leq c(0) \\
        -\infty \  \text{ in other case}
    \end{cases} \notag
\end{align}

Where

\begin{equation*}
    h_j(y)= \begin{cases}
        0, \text{ if } \frac{y_j}{\theta} >0 \\
        -\frac{1}{2\theta} y_j^2 , \text{ if } \frac{y_j}{\theta} \in[-p_j,0] \\

        \frac{1}{2}\theta p_j^2 + y_jp_j, \text{ if } \frac{y_j}{ \theta} < -p_j
    \end{cases}
\end{equation*}

Finally,

\begin{equation}
\label{argmax psi3.2}
    \argmax \psi_3(y,z)= \Pi_{j=1}^m Y_j \times \argmax \{  z^Tb, \text{ s.t. } A^Tz \leq c(0), z \leq 0\}, \text{with } Y_j=\begin{cases}
    \R^m_+, \text{ if } p_j \neq 0, \\
   \R, \text{ if } p_j=0.
   \end{cases}
\end{equation}

\begin{remark}
    Similarly, the dual function of \eqref{AR3}, when $\theta\nup =0$, is given by  $\psi_3\nup (y,z) = z^T b+ \sum_{j \in J} p_j\nup y_j - \iota_{\{A^Tz \leq c\nup(0)\}}(z) - \iota_{ \{z\leq 0\} }(z)$, where $c\nup(u) \in\R^n$ is the vector, given blockwise by
\begin{equation*}
    c_i\nup(u) = \begin{cases}
        0, & \forall i \in I_t, ~\forall t\in\{ 1,...,T-1\}, \\
        -\sum_{j \in i^+} (p_j\nup+u_j) S_{T+1}(j), &\forall i \in I_T.
    \end{cases}
\end{equation*}
    In the case $\theta\nup >0$, the dual function is $\psi_3\nup (y,z)=z^T b+ \sum_{j =1}^m h_j\nup(y) - \iota_{\{A^Tz \leq c\nup(0)\}}(z) - \iota_{ \{z\leq 0\} }(z),$
    where
    \begin{equation*}
    h_j\nup(y)= \begin{cases}
        0, \text{ if } \frac{y_j}{\theta} >0 \\
        -\frac{1}{2\theta} y_j^2 , \text{ if } \frac{y_j}{\theta} \in[-p_j\nup,0] \\

        \frac{1}{2}\theta p_j^{\nu2} + y_jp_j\nup, \text{ if } \frac{y_j}{ \theta} < -p_j\nup
    \end{cases}
\end{equation*}
\end{remark}

\begin{remark}
    Note that the dual functions of the Rockafellian \eqref{R3} and the Approximating Rockafellian \eqref{AR3} require the constraint $u \leq 0$. Otherwise, since $c(u)$ is decreasing, there exists some $u^* \in \mathbb{R}^m$ such that $(c(u^*) - A^T z) \leq 0$ and letting $X$ tend to infinity in \eqref{*} for the Rockafellian perturbation \eqref{R3} with $\theta =0$,  or \eqref{**} in the case $\theta >0$, the dual functions become degenerate, that is, $\psi_3 \equiv -\infty$.
\end{remark}

\subsubsection{Tightness of Approximating Rockafellian}

In this section, we are interested in studying the tightness of the approximating tilted Rockafellian
$$f\nup_{y\nup,z\nup}(u,v,X):= f_3\nup(u,v,X) - \langle y\nup,u \rangle - \langle z\nup, v\rangle$$
with $f_3\nup$ defined in~\eqref{AR3}

\begin{lemma}[Tightness of $\{f^\nu_{y^\nu,z^\nu}\}$]
    Let $\{z\nup\} \subset \R^{N_T}$ be a sequence such that $z\nup \leq 0$ and $A^Tz\nup \leq c\nup(0)$ for $\nu \geq \nu_0$, let $\{y\nup\} \subset \R^m$ a sequence and let $\{f\nup_{y\nup,z\nup}\}_{\nu \in \N}$ the tilted Approximating Rockafellian of $f\nup$ \eqref{AR3}. Then the family $\{f\nup_{y\nup,z\nup}\}_{\nu \in \N}$ is tight.
\end{lemma}

\begin{proof}
    Let $\varepsilon >0$ and the sequence $\{z\nup\}_{\nu \in \N}$ satisfies $z\nup \leq 0$ and $A^Tz\nup \leq c\nup(0)$ for $\nu \geq \nu_0$, we define $B_\varepsilon = \{-b\} \times [-1,0]^m \times \{0\}$ (independent of $\varepsilon$), then
    \begin{align*}
        \inf_{(u,v,X) \in B_\varepsilon} f\nup_{y\nup,z\nup}(u,v,X) &= {z\nup}^Tb + \inf_{-p\nup \leq u \leq 0}\frac{1}{2}\theta\nup\|u\|^2  - \langle y\nup,u \rangle\\
        &= \inf_{(u,v,X) \in \R^m \times \R^{N_T}\times \R^n} f\nup_{y\nup,z\nup}(u,v,X) \\
        &\leq \inf_{(u,v,X) \in \R^m \times \R^{N_T}\times \R^n} f\nup_{y\nup,z\nup}(u,v,X) + \varepsilon \ \ \forall \nu \geq\nu_0
    \end{align*}
    Note that this proof works regardless of whether $\theta\nup > 0$ or $\theta\nup = 0$.
\end{proof}

\subsubsection{Hypo-convergence of dual functions}

\begin{theorem}[Hypo-convergence of Dual Functions]\label{thm:hypo3}
    If $\theta\nup$ converge to $\theta$ and $p\nup$ converge to $p$, then the dual function of \eqref{AR3}  hypo-converge to the dual function of \eqref{R3}. Furthermore, for any sequence $\{y\nup,z\nup\}$ that converges to $y,z \in$ dom$(-\psi_3)$, we have that $\psi_3\nup(y\nup,z\nup) \rightarrow \psi_3(y,z)$.
\end{theorem}
\begin{proof}
    Note that since $p\nup \to p$ and $\theta\nup \to \theta$ we have that the epi-convergence of $f\nup$ to $f$ is given by Propositions \ref{prop:epi3}. Let $(y,z) \in$ dom$(-\psi)$, since $\{f\nup_{y\nup,z\nup}\}_{\nu \in \mathbb{N}}$ is tight for sequence $z\nup$ such that $z\nup \leq 0$ and $A^Tz\nup \leq c\nup(0)$ for $\nu \geq \nu_0$ and $c\nup(0)$ converge to $c(0)$, then exist $z\nup$ such that converge to $z$ (the proof of the existence of this sequence is identical to the one given in Lemma \ref{lem:interior-d}) and we can chose any sequence $\{y\nup\}$ that converge to $y$. Then, $\{f\nup_{y\nup,z\nup}\}$ is tight for the sequence $\{y\nup,z\nup\}_{\nu \in \N}$. Finally, using \cite[Thm.~5.1]{deride2024approximations} the result follows.
\end{proof}

\subsubsection{Strong duality}

\begin{theorem}[Strong duality]\label{thm:SD3}
The dual value of \eqref{R3} coincides with the primal value, regardless of whether $\theta = 0$ or $\theta > 0$.
\end{theorem}

\begin{proof}
    By Theorem \eqref{thm:SD2}, the problem
    \begin{align*}
        &\max_{z \in \mathbb{R}^{N_T}} z^T b \\
        &\text{s.t. } A^T z \leq c(0), \quad z \leq 0
    \end{align*}
    has the same optimal value as the primal problem.
    Recalling that $\argmax \psi_3$ is characterized by~\eqref{argmax psi3.1} for $\theta=0$ and by~\eqref{argmax psi3.2} for $\theta>0$, with both characterizations being equivalent. Hence, for any $y \in \mathbb{R}_+^m$, the value of the dual function $\psi_3(y,z)$ remains unchanged. Therefore, the dual optimal value coincides with the primal optimal value.
\end{proof}

\begin{proposition}\label{prop:dual tight 2}
    Suppose that $p\nup \rightarrow p$. Then the approximating dual function $-\psi_3\nup(y,z)$ is tight.
\end{proposition}

\begin{proof}
    Similar as in Proposition~\ref{prop:dual tight}, the problem $\max_{z \in \mathbb{R}^{N_T}} z^T b \text{ s.t. } A^T z \leq c^\nu(0), z \leq 0$ has the same Lipschitz properties in the set of optimizers. Then, since $p^\nu \to p$, we have $c^\nu(0) \to c(0)$, there exists a compact set $B$ such that, for any $y \in \mathbb{R}^m_+$, $-\psi_3^\nu(y,z)$ is tight.
\end{proof}

\begin{remark}
    By Proposition~\ref{prop:dual tight 2} and Theorem~\ref{thm:hypo3}, we have that the maximizers of the approximate dual problems converge to the original ones.
\end{remark}

\subsubsection{Economic interpretation}

The interpretation of the dual variables is similar to that presented previously, since $\argmax \psi_3(y,z)$
can be decomposed as $\R^m_+ \times \argmax \psi(z),$ where $z$ corresponds to the dual variable associated with the claim perturbation. Therefore, in both cases ($\theta =0$ and $\theta >0$),  the information provided by $y$ is the same as in case~2.2 of Section~\ref{sec:4.1.3}, while the information provided by $z$ coincides with that described in Section~\ref{sec:4.2.5}.

\section{Supermartingale property and equivalent dual formulation}\label{sec:5}

Consider the classical Lagrangian associated with the claim perturbation introduced in Section~\ref{sec:3.2}, with the duality pairing defined by $\langle a,b \rangle = \mathbb{E}[ab]$. To ensure that the duality pairing is well-defined, we consider perturbations $u \in L^p(\Omega,\mathcal{A},\mathbb{P})$ and multipliers $\lambda \in L^q(\Omega,\mathcal{A},\mathbb{P})$, where $p$ and $q$ are conjugate exponents satisfying $\frac{1}{p}+\frac{1}{q}=1$. The Lagrangian is then given by
\begin{align*}
l(X,\lambda,\mu)
={}& -\mathbb{E}\big[\langle S_{T+1},X_T\rangle\big]
+ \lambda_0\big(\langle S_0,X_0\rangle-G_0\big) \\
&\qquad\qquad+ \sum_{t=1}^T
\mathbb{E}\big[
\lambda_t\big(\langle S_t,X_t-X_{t-1}\rangle-G_t\big)
\big]
- \sum_{t=0}^T
\mathbb{E}\big[\langle\mu_t,X_t\rangle\big].
\end{align*}
Let us examine the necessary conditions for the infimum over $X$ to be finite. Isolating the terms involving $X_0$ yields
\begin{align*}
\inf_{X_0 \in \R} \mathbb{E}\big[\langle \lambda_0 S_0 - \mu_0 - \mathbb{E}[\lambda_1 S_1 \mid \mathcal{F}_0], X_0 \rangle\big].
\end{align*}
For this infimum to be finite, the term inside the inner product must vanish almost surely. Since $\mu_0 \geq 0$, we have $\mathbb{E}[\lambda_1 S_1 \mid \mathcal{F}_0] \leq \lambda_0 S_0$. Analogously, optimizing over $X_t$ for $t\in\{ 1, \dots, T-1\}$ yields the conditions $\mathbb{E}[\lambda_{t+1}S_{t+1} \mid \mathcal{F}_t] \leq \lambda_t S_t$. Optimizing over $X_T$ results in $\mathbb{E}[S_{T+1} \mid \mathcal{F}_T] \leq \lambda_T S_T$. Consequently, the dual problem associated with the Lagrangian is given by
\begin{align}
    \max_{\lambda \geq 0} \quad & -\lambda_0 G_0 - \sum_{t=1}^T \mathbb{E}[\lambda_t G_t] \label{dual problem lp} \\
    \text{s.t.} \quad & \mathbb{E}[\lambda_{t+1} S_{t+1} \mid \mathcal{F}_t] \leq \lambda_t S_t, \quad \forall t \in \{0, \dots, T-1\}, \label{restriction lp} \\
    & \mathbb{E}[S_{T+1} \mid \mathcal{F}_T] \leq \lambda_T S_T. \label{restriction final}
\end{align}

We now derive an equivalent formulation of the dual problem \eqref{dual problem lp} in terms of risk-neutral measures. This characterization enables us to establish a connection with the mathematical finance literature, specifically with the Fundamental Theorem of Asset Pricing under short-selling prohibitions \cite{pulido2014fundamental}, following an approach similar to that of \cite{Dahl}. Before doing so, we present a couple of definitions.

\begin{definition}
 A discounted market is a market where there exists an asset whose value is equal to one at every time $t \in \{0,\ldots,T+1\}$ and for every possible scenario. Since $S^0$ is a risk--free asset, the normalization is obtained by dividing  the risky assets by the value of $S^0_t$ for each period $t$, thus obtaining:
\[
S_t^0(\xi_{0:t}) = 1,
\qquad
\forall t \in \{0,\ldots,T+1\},\ \forall \xi_{0:t}.
\]

\end{definition}

\begin{definition}
Let $\mathcal{M}$ be the set of finite measures defined by
\begin{equation*}
    \mathcal{M} := \left\{ Q \sim \mathbb{P} : Q(\Omega) < \infty, \, \{S_t\}_{t=0}^T \text{ is a } \tilde{Q}\text{-supermartingale}, \text{ and } \frac{dQ}{d\mathbb{P}}S_T \geq \mathbb{E}[S_{T+1} \mid \mathcal{F}_T] \right\},
\end{equation*}
where $\tilde{Q}$ is the normalized probability measure defined by $\tilde{Q}(A) = \frac{Q(A)}{Q(\Omega)}$ for all $A \in \mathcal{F}$.
\end{definition}

\begin{theorem}[Equivalent supermartingale-measure formulation]\label{thm:FTAP}
Consider a discounted market. Assume that the optimal portfolio satisfies $X_t^0 > 0$ almost surely for all $t \in \{0, \dots, T\}$ (i.e., the risk--free asset is strictly held in the portfolio). Then, the dual problem \eqref{dual problem lp}-\eqref{restriction final} is equivalent to
\begin{equation}
        \sup_{Q \in \mathcal{M}} -\sum_{t=0}^T \mathbb{E}^{\tilde{Q}}[G_t] \cdot Q(\Omega). \label{Qmartingale}
    \end{equation}
\end{theorem}

\begin{remark}
As noted in the modeling discussion, the condition $X_t^0 > 0$ is not an explicit restriction hardcoded into the dual optimization problem itself. Instead, it is an interiority condition on the primal optimal allocation. By complementary slackness, this condition nullifies the corresponding Lagrange multipliers for the short-selling constraint on the risk--free asset ($\mu_t^0 = 0$), which structurally forces the dual state-price density process $(\lambda_t)_{t=0}^T$ to be a pure $\mathbb{P}$-martingale rather than a supermartingale.
\end{remark}

\begin{proof}
(i) Let $Q \in \mathcal{M}$. We construct a vector $\lambda$ that is feasible for the dual problem \eqref{dual problem lp}--\eqref{restriction final}. Let $\lambda_T := \frac{dQ}{d\mathbb{P}}$ be the Radon-Nikodym derivative and $\lambda_t = \mathbb{E}[\lambda_T \mid \mathcal{F}_t]$ for $t\in\{ 0, \dots, T-1\}$. Then $\frac{d\tilde{Q}}{d\mathbb{P}} = \frac{\lambda_T}{\mathbb{E}[\lambda_T]} =: \nu_T$ and $\nu_t := \frac{\lambda_t}{\mathbb{E}[\lambda_T]}$. Since $Q \in \mathcal{M}$, constraint \eqref{restriction final} holds by definition. Let us now verify that $(\lambda_t)_{t=0}^T$ satisfies the constraints in \eqref{restriction lp}.

For $t \in\{ 0, \dots, T-1\}$, we have to prove
\begin{equation}
\mathbb{E}[\lambda_{t+1}S_{t+1} \mid \mathcal{F}_{t}] \leq \lambda_{t}S_{t} = \mathbb{E}[\lambda_{t}S_{t} \mid \mathcal{F}_{t}]. \label{eq4}
\end{equation}
For this, note that by Bayes' rule for conditional expectations under a change of measure:
\begin{equation*}
\mathbb{E}[\nu_{t+1}S_{t+1} \mid \mathcal{F}_{t}] = \mathbb{E}[\nu_T \mid \mathcal{F}_t] \mathbb{E}^{\tilde{Q}}[S_{t+1} \mid \mathcal{F}_t],
\end{equation*}
which, upon multiplying both sides by $\mathbb{E}[\lambda_T]$, is equivalent to
\begin{equation}
\mathbb{E}[\lambda_{t+1}S_{t+1} \mid \mathcal{F}_{t}] = \mathbb{E}[\lambda_T \mid \mathcal{F}_t] \mathbb{E}^{\tilde{Q}}[S_{t+1} \mid \mathcal{F}_t]. \label{eq5}
\end{equation}
On the other hand, we have
\begin{equation*}
\mathbb{E}[\nu_{t}S_{t} \mid \mathcal{F}_{t}] = \mathbb{E}[\nu_T \mid \mathcal{F}_t] \mathbb{E}^{\tilde{Q}}[S_t \mid \mathcal{F}_t].
\end{equation*}
Then, multiplying by $\mathbb{E}[\lambda_T]$, we obtain
\begin{equation}
\mathbb{E}[\lambda_{t}S_{t} \mid \mathcal{F}_{t}] = \mathbb{E}[\lambda_T \mid \mathcal{F}_t] \mathbb{E}^{\tilde{Q}}[S_t \mid \mathcal{F}_t]. \label{eq6}
\end{equation}
Combining \eqref{eq5} and \eqref{eq6}, proving \eqref{eq4} is equivalent to proving
\begin{equation*}
\mathbb{E}[\lambda_T \mid \mathcal{F}_t] \mathbb{E}^{\tilde{Q}}[S_{t+1} \mid \mathcal{F}_t] \leq \mathbb{E}[\lambda_T \mid \mathcal{F}_t] \mathbb{E}^{\tilde{Q}}[S_t \mid \mathcal{F}_t],
\end{equation*}
which always holds since $\lambda \geq 0$ and $Q \in \mathcal{M}$, meaning $\{S_t\}_{t=0}^T$ is a $\tilde{Q}$-supermartingale.

(ii) Now, let us prove that if $\lambda$ is optimal solution for problem \eqref{dual problem lp}--\eqref{restriction final}, then there exists a measure $Q \in \mathcal{M}$.

Let $Q(A) = \int_{A} \lambda_T \, d\mathbb{P}$ for all $A \in \mathcal{F}$, then the condition $\frac{dQ}{d\mathbb{P}}S_T \geq \mathbb{E}[S_{T+1} \mid \mathcal{F}_T]$ holds by definition.

Let us prove that, for all $t \in \{0, \dots, T-1\}$ the following holds
\begin{equation*}
    \mathbb{E}^{\tilde{Q}}[S_{t+1} \mid \mathcal{F}_t] \leq \mathbb{E}^{\tilde{Q}}[S_t \mid \mathcal{F}_t] = S_t.
\end{equation*}
Note that, for all $A \in \mathcal{F}_t$
\begin{align}
    \int_{A} \lambda_t S_t \, d\mathbb{P} &\geq \int_A \lambda_{t+1} S_{t+1} \, d\mathbb{P} \nonumber \\
    &= \int_A \mathbb{E}[\lambda_{t+1} S_{t+1} \mid \mathcal{F}_t] \, d\mathbb{P} \nonumber \\
    &= \int_A \mathbb{E}\big[\mathbb{E}[\lambda_{T} \mid \mathcal{F}_{t+1}] S_{t+1} \mid \mathcal{F}_t\big] \, d\mathbb{P} \nonumber \\
    &= \int_A \mathbb{E}\big[\mathbb{E}[\lambda_{T} S_{t+1} \mid \mathcal{F}_{t+1}] \mid \mathcal{F}_t\big] \, d\mathbb{P} \nonumber \\
    &= \int_A \mathbb{E}[\lambda_{T} S_{t+1} \mid \mathcal{F}_t] \, d\mathbb{P} = \int_A \lambda_t \mathbb{E}^{\tilde{Q}}[ S_{t+1} \mid \mathcal{F}_t] \, d\mathbb{P}. \label{martingala de l}
\end{align}
Where the last equality in \eqref{martingala de l} follows by the condition $X_t^0 > 0$ for all $t \in \{0, \dots, T\}$ together with complementary slackness; from this, we obtain that $\mathbb{E}[\lambda_{t+1} \mid \mathcal{F}_t] = \lambda_t$, which means that $(\lambda_t)_{t=0}^T$ is a true $\mathbb{P}$-martingale. Then,
\begin{equation*}
\int_A \lambda_t \big( S_t - \mathbb{E}^{\tilde{Q}}[ S_{t+1} \mid \mathcal{F}_t] \big) \, d\mathbb{P} \geq 0.
\end{equation*}
Since $A \in \mathcal{F}_t$ is arbitrary, if $\lambda_t > 0$ it follows that
\begin{equation*}
S_t - \mathbb{E}^{\tilde{Q}}[S_{t+1} \mid \mathcal{F}_t] \geq 0 \quad \mathbb{P}\text{-a.s.},
\end{equation*}
and note that $\lambda_t \neq 0$ a.s., since the market is discounted and from \eqref{restriction final}, we obtain that $\lambda_T \geq 1$, and by \eqref{restriction lp} we obtain that $(\lambda_t)_{t=0}^T \geq 1$.
\end{proof}

The result obtained here is analogous to that in \cite{Dahl}. The primary distinction lies in the modeling perspective: while \cite{Dahl} formulates the problem from the viewpoint of the contingent claim's writer, the characterization in Theorem~\ref{thm:FTAP} is derived from the perspective of an investor seeking to maximize expected terminal wealth.

Furthermore, this characterization provides a bridge between our duality framework and the Fundamental Theorem of Asset Pricing. Because a no-short-selling constraint is imposed on the portfolio, the absence of arbitrage is known to be equivalent to the existence of at least one supermartingale measure (see \cite{pulido2014fundamental}). In our setting, these pricing measures are explicitly determined by the optimal dual variables of problem \eqref{dual problem lp}--\eqref{restriction final}.

\begin{remark}
The interiority condition $X_t^0 > 0$ almost surely for all $t \in \{0, \dots, T\}$ in Theorem~\ref{thm:FTAP} is superfluous if short selling of the risky assets is permitted. In an unconstrained market, the inequalities in the dual constraints \eqref{restriction lp} and \eqref{restriction final} tighten into equalities. Consequently, the normalized price process $\{S_t\}_{t=0}^T$ becomes a true $\tilde{Q}$-martingale, and the proof of Theorem~\ref{thm:FTAP} carries through directly upon replacing the respective inequalities with equalities.
\end{remark}

\section{Numerical examples} \label{sec:6}

\subsection{Failure of epi-convergence} \label{sec:counterexample}

In Section~\ref{sec:counterproposition}, we showed that the convergence of the underlying probability measures alone does not guarantee the stability of the associated Rockafellian functions, if other components of the market are perturbed at the same time. We now illustrate this phenomenon with an explicit example in which, although the perturbed probabilities converge to the original distribution, the corresponding optimal solutions fail to converge.

The perturbation consists of introducing catastrophic events with vanishing probabilities but sufficiently large asset prices to alter the investor's optimal decisions. Despite becoming increasingly unlikely, these events have a persistent effect on the objective function, causing the investor to switch from a fully risk--free portfolio to one invested entirely in the risky asset. This example provides a concrete illustration of the instability proved in Proposition~\ref{prop:no epi}.

Consider a market with two assets and horizon $T=2$. For $t\in\{0,1,2,3\}$, the random variables $\xi_t$ are i.i.d. and distribute as $\xi_t \sim \text{Ber}(\{-1,1\},p)$. The dynamics of the assets are the following
\begin{itemize}
    \item $S^0_t$: risk--free asset with dynamics $S^0_{t+1}=(1+r)S_t^0$.
    \item $S^1_t$: risky asset (stock) with dynamics $S_{t+1}^1(\xi_{0:t+1}) = S_t^1(\xi_{0:t}) + \mu + \sigma \xi_{t+1}$.
\end{itemize}

We consider the numerical parameters in Table~\ref{tab:model-parameters}.
\begin{table}[tbhp]
    \centering
    \caption{Model parameters for epi-convergence failure example (Ex.~\ref{sec:counterexample}).}
    \label{tab:model-parameters}
    \begin{tabular}{
        S[table-format=1.3]
        S[table-format=1.1]
        S[table-format=1.2]
        S[table-format=2.2]
        S[table-format=1.0]
        S[table-format=4.2]
        S[table-format=6.0]
        S[table-format=-5.0]
        S[table-format=-5.0]
    }
        \toprule
        {$r$} & {$p$} & {$\mu$} & {$\sigma$}
        & {$S_0^0$} & {$S_0^1$} & {$G_0$} & {$G_1$} & {$G_2$} \\
        \midrule
        0.005 & 0.5 & 9.91 & 91.03
        & 1 & 8041.14 & 100000 & -50000 & -50100 \\
        \bottomrule
    \end{tabular}
\end{table}

The stock-price tree generated by these parameters is shown in Figure~\ref{fig:stock-prices}.

\begin{figure}[t]
\centering
\includegraphics[width=.75\textwidth]{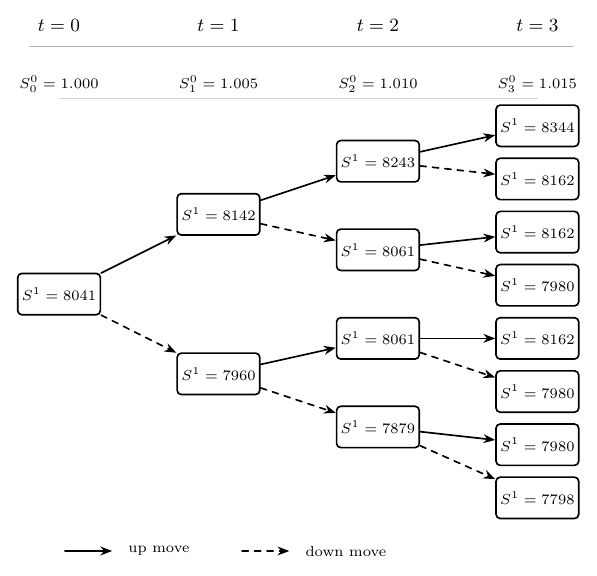}
\caption{Stock-price tree for the risky asset, together with the deterministic evolution of the risk--free asset. Solid branches denote up moves and dashed branches denote down moves.}
\label{fig:stock-prices}
\end{figure}

The solution to the investment problem, that is the minimizer of $\varphi(X) = - \mathbb{E} [\langle S_{T+1},X_T \rangle] + \iota_{\mathcal{X}}(X)$ is shown in Figure~\ref{fig:solution-example}. The results indicate  that the risky asset is unattractive and all the money is invested in the risk--free asset.

\begin{figure}[t]
\centering
\includegraphics[width=.75\textwidth]{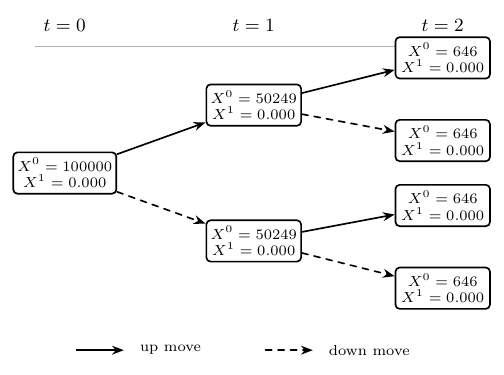}
\caption{Optimal allocation in the initial problem.}
\label{fig:solution-example}
\end{figure}

We introduce now a catastrophic event. From each node at time $t=2$, we add an additional successor node. If the probability of a parent node is $p_j$, the catastrophic successor is assigned probability $p_j q^\nu$. The remaining probability, $p_j(1-q^\nu)$, is distributed among the two original successors nodes proportionally to their original probabilities. At the catastrophic node, the value of the risky asset $S^1$ is multiplied by the factor $\kappa\nup = (1+ \kappa / q\nup)$, where $q\nup$ is any sequence converging to $0$ and $\kappa \geq 0.0075$. Figure~\ref{fig:perturbed-market} illustrates the resulting tree for $q^\nu = 0.001$ and $\kappa = 0.0075$.

\begin{figure}[t]
\centering
\includegraphics[width=.75\textwidth]{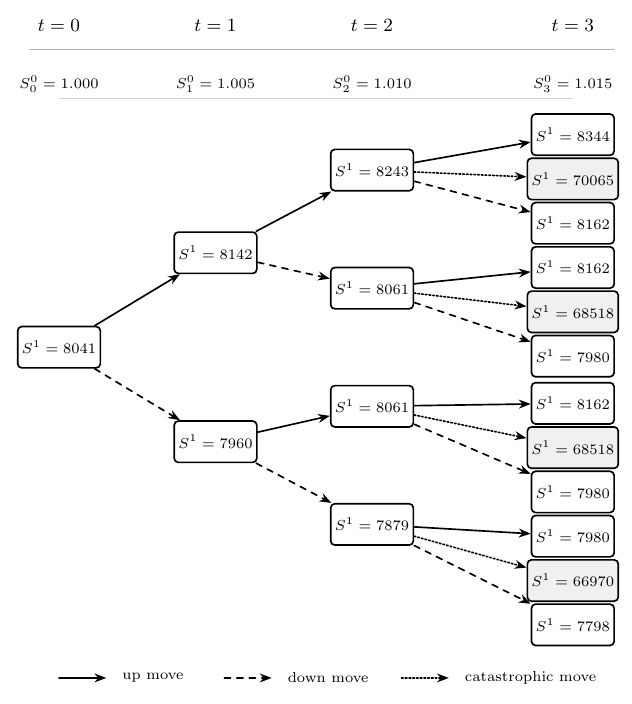}
\caption{Perturbed market tree for $q^\nu=0.001$ and $\kappa=0.0075$. Dotted branches denote catastrophic successors.}
\label{fig:perturbed-market}
\end{figure}

We define the perturbed problem as the one minimizing $\varphi\nup(X )=  -\mathbb{E}^{p\nup}[\langle S_{T+1},X_T \rangle] + \iota_{\mathcal{X}}(X)$, where $p^\nu$ is the probability as in Remark  \ref{rem:definition-p}, after the introduction of the catastrophic nodes. Observe that the `original' problem of minimizing $\varphi(X)$ is recovered if $q^\nu=0$, given the convention $0\cdot\infty=0$\footnote{To be more precise, the original problem is equivalent to the perturbed one with $q^\nu=0$.}. Therefore, the question of interest lies in the convergence of the minimizer $X^\nu$ of $\varphi\nup$ towards the minimizer $X$ of $\varphi$.

Before the perturbation, the coefficient of $X_T^1(j)$ in the objective function is
\[
\sum_{i\in j^+}\frac{1}{8}S_{T+1}^1(i)
=\frac{1}{4}\bigl(S_T^1(j)+\mu\bigr).
\]
After introducing the catastrophic nodes, this coefficient becomes
\[
\frac{1}{4}\bigl(S_T^1(j)(1+\kappa)+\mu\bigr)
+\frac{q^\nu\mu}{4}.
\]
Thus, the expected payoff associated with the risky asset increases at every node, creating an incentive to invest more heavily in it. Consequently, the optimal solutions differ significantly from those of the original problem, yielding the behavior shown in Figure~\ref{fig:solution-perturbation} for every $\nu\in\mathbb{N}$.

\begin{figure}[t]
\centering
\includegraphics[width=.72\textwidth]{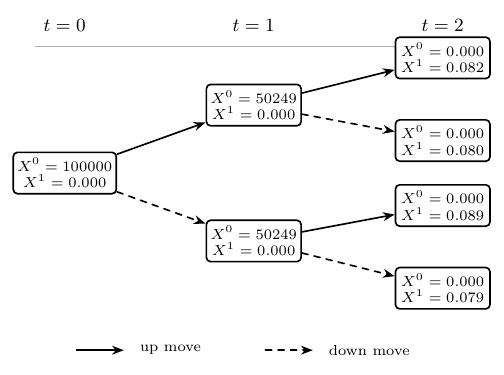}
\caption{Optimal allocation in the perturbed problem for any $\nu\in\mathbb{N}$.}
\label{fig:solution-perturbation}
\end{figure}

We emphasize that, although the perturbed market converges to the original one, the corresponding optimal solutions do not converge to an optimal solution of the original problem.

Since $\lim_{\nu\to\infty} q^\nu\kappa^\nu=\kappa>0$, it follows from Proposition~\ref{prop:no epi} that the Rockafellian associated with the perturbed problem does not epi-converge to the Rockafellian of the original problem. Moreover, in this example, the sequence of optimal solutions of the perturbed problems fails to converge to an optimal solution of the original problem.

This example highlights the limitations of the stability framework based on Rockafellian perturbations. By introducing catastrophic events with vanishing probabilities but unbounded stock prices, one obtains a sequence of optimization problems whose underlying market models converge to the original one, while their optimal solutions remain unstable. At the same time, the associated Rockafellian functions fail to be stable in the sense of epi-convergence.

\subsection{Dual stability}\label{ex:dual}

Now, we analyze the convergence of the dual values associated with claim perturbations defined in Section~\ref{sec:3.2}. To this end, we consider the same market model as before, with the parameters defined in Table~\ref{tab:model-parameters1}.

\begin{table}[tbhp]
    \centering
    \caption{Model parameters for the dual stability example (Ex.~\ref{ex:dual}).}
    \label{tab:model-parameters1}
    \begin{tabular}{
        S[table-format=1.1]
        S[table-format=3.2]
        S[table-format=2.2]
        S[table-format=1.0]
        S[table-format=4.2]
        S[table-format=6.0]
        S[table-format=-5.0]
        S[table-format=-5.0]
    }
        \toprule
        {$p$} & {$\mu$} & {$\sigma$}
        & {$S_0^0$} & {$S_0^1$}
        & {$G_0$} & {$G_1$} & {$G_2$} \\
        \midrule
        0.5 & 109.91 & 91.03
        & 1 & 8041.14
        & 100000 & -50000 & -50100 \\
        \bottomrule
    \end{tabular}
\end{table}

As in Section~\ref{sec:3}, we consider a critical event in which the value of the risky asset is multiplied by $k = 0.5$ with probability $q\nup$.

For this example, we consider the values
$q\nup \in \left\{0,\;10^{-2},\;10^{-3},\;10^{-4},\;10^{-5}\right\}.$ and with $q\nup=0$ we recover the original case.

The corresponding dual solutions are depicted in Figure~\ref{fig:dual-solution-convergence}.  In this example, the approximate dual solutions converge to the original dual solution, as illustrated in Figure~\ref{fig:dual-solution-convergence}, by Remark~\ref{rem:convergencia de soluciones} implies that
$\sup \psi_2^\nu \longrightarrow \sup \psi_2.$ As a consequence, the dual problem is stable with respect to critical event, recovering the original solution when the approximate probabilities converge to the original probability.

\begin{figure}[t]
\centering
\includegraphics[width=0.95\textwidth]{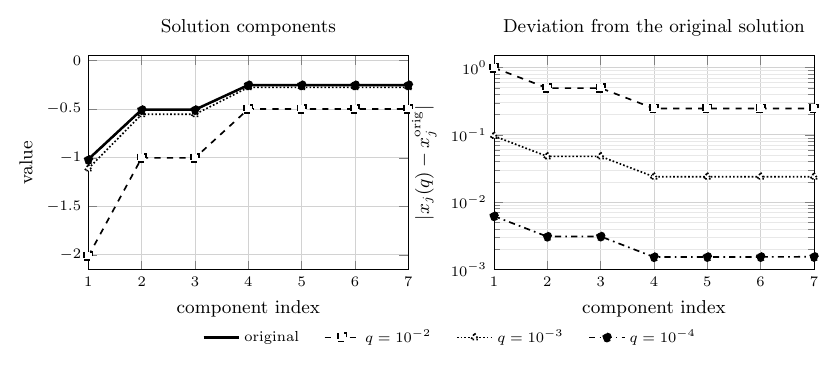}
\caption{Convergence of the perturbed solution to the original solution. The left panel reports the seven solution components for the original model and for perturbation levels $q\in\{10^{-2},10^{-3},10^{-4}\}$. The right panel reports the componentwise absolute deviation from the original solution on a logarithmic scale.}
\label{fig:dual-solution-convergence}
\end{figure}

\section{Conclusions and further research}

In this work, we provide a new framework for analyzing the stability and robustness of discrete-time contingent-claim problems over a finite horizon. We use Rockafellian perturbations to embed the wealth maximization problem of the writer into a family of perturbed models and show the stability of this class of perturbations. We provided theoretical guarantees for the epi-convergence of tilted Rockafellians and the hypo-convergence of dual functions, ensuring that dual approximate solutions converge to the real solution under conditions of tightness and convergence of probability. The dual analysis provides a rich economic interpretation of the multipliers by viewing them as shadow prices, while also establishing a connection with the Fundamental Theorem of Asset Pricing (FTAP).

Future work has different edges. The first potential line of research is on relaxing the no-short-selling condition. The no-short-selling condition was enforced to ensure the tightness of approximating tilted Rockafellians and in order for the dual problem to be well defined. Exploring alternate regularity conditions to accommodate short positions without inducing model collapse would significantly enlarge the variety of solvable problems. The second line consists of extending this framework to a continuous-time setting, transitioning from finite discrete support to prices driven by semimartingales. Finally, a further avenue for future research is to consider alternative objective functions incorporating risk aversion, such as the Conditional Value-at-Risk (CVaR) or expected gain-loss criteria.
\appendix
\section{Calculations for the probability-perturbation dual}\label{A}

Case \eqref{case:1.1}:  Let $X \in \mathcal{X}$ and $p+u \in [0,+ \infty[^m$, in other case, the Lagrangian $l(X,y)= \infty$ and let $\theta = 0$, then
\begin{align*}
    \inf_{u} f(u,X) - \langle u,y\rangle &= \inf_u \sum_{j=1}^m (p_j+u_j) \langle-S_{T+1}(j),X_T(j^-)\rangle -u_jy_j \\
    &= \inf_u \sum_{j=1}^m p_j \langle-S_{T+1}(j),X_T(j^-)\rangle +u_j(\langle-S_{T+1}(j),X_T(j^-)\rangle -y_j)
\end{align*}

Consequently, if $\langle-S_{T+1}(j),X_T(j^-)\rangle -y_j <0$ for any $j$, then taking $u_j \rightarrow \infty$, we have $h_j(X)= -\infty$, and if $\langle-S_{T+1}(j),X_T(j^-)\rangle -y_j \geq 0$ we take $u_j=-p_j$ and $h_j(X,y)= p_jy_j$. Then

$$h_j(X,y) = \begin{cases}
    p_jy_j \text{, if }\langle-S_{T+1}(j),X_T(j^-)\rangle -y_j \geq 0 \\
    -\infty \text{, in other case}
\end{cases}$$

Case \eqref{case:2.1}: In the case, when $\langle-S_{T+1}(j),X_T(j^-)\rangle -y_j <0$, the minimum is reached at $u=0$, and
$$h_j(X,y) = \begin{cases}
    p_jy_j, &\text{ if }\langle-S_{T+1}(j),X_T(j^-)\rangle -y_j \geq 0 \\
     p_j\langle-S_{T+1}(j),X_T(j^-)\rangle, &\text{ in other case}
\end{cases}$$

Case \eqref{case:1.2}
\begin{align*}
    \inf_{u} f(u,X) - \langle u,y\rangle &= \inf_{-p \leq u} \sum_{j=1}^m (p_j+u_j) \langle-S_{T+1}(j),X_T(j^-)\rangle -u_jy_j + \frac{1}{2}\theta u_j^2 \\
    &= \inf_{-p \leq u} \sum_{j=1}^m p_j \langle-S_{T+1}(j),X_T(j^-)\rangle +u_j(\langle-S_{T+1}(j),X_T(j^-)\rangle -y_j) + \frac{1}{2} \theta u_j^2
\end{align*}

Let $b_j := \langle-S_{T+1}(j),X_T(j^-)\rangle -y_j$, then

If $\frac{-b_j}{\theta} \geq -p_j$, then $u_j=\frac{-b_j}{\theta}$ and $h_j(X,y)= p_j\langle-S_{T+1}(j), X_T(j^-)\rangle - \frac{1}{2\theta} (y_j - \langle -S_{T+1}(j), X_T(j^-)\rangle)^2$.

If $\frac{-b_j}{\theta} < -p_j$, the minimum is attained at $u_j = -p_j$ and $h_j(X,y)= \frac{1}{2} \theta p_j^2 + y_jp_j$

Then
$$h_j(X,y) = \begin{cases}
    \frac{1}{2} \theta p_j^2 + y_jp_j, & \text{if } \langle-S_{T+1}(j), X_T(j^-)\rangle \in ]y_j+ \theta p_j, +\infty] \\[1ex]
    \begin{aligned}
        & p_j\langle-S_{T+1}(j), X_T(j^-)\rangle \\
        & - \frac{1}{2\theta} (y_j + \langle S_{T+1}(j), X_T(j^-)\rangle)^2,
    \end{aligned} & \text{if } \langle-S_{T+1}(j), X_T(j^-)\rangle \in ]-\infty, y_j + \theta p_j]
\end{cases}$$

Case~\eqref{case:2.2}

If $-p_j \leq \frac{-b_j}{\theta} \leq 0$, $h_j(X)=p_j\langle-S_{T+1}(j), X_T(j^-)\rangle - \frac{1}{2\theta} (y_j - \langle -S_{T+1}(j), X_T(j^-)\rangle )^2$, the case $\frac{-b_j}{\theta} <-p_j$ is the same as the previous one, and finally the case $\frac{-b_j}{\theta} >0$ is minimized by $u=0$ and $h_j(X)=p_j \langle-S_{T+1}(j), X_T(j^-)\rangle $

Then

$$h_j(X,y) = \begin{cases}
    \frac{1}{2} \theta p_j^2 + y_jp_j, & \text{if } \langle-S_{T+1}(j), X_T(j^-)\rangle \in ]y_j+ \theta p_j, +\infty] \\[1.5ex]
    \begin{aligned}
        & p_j\langle-S_{T+1}(j), X_T(j^-)\rangle \\
        & - \frac{1}{2\theta} (y_j + \langle S_{T+1}(j), X_T(j^-)\rangle)^2,
    \end{aligned} & \text{if } \langle-S_{T+1}(j), X_T(j^-)\rangle \in [y_j, y_j + \theta p_j] \\[1.5ex]
    p_j \langle-S_{T+1}(j), X_T(j^-)\rangle, & \text{if } \langle-S_{T+1}(j), X_T(j^-)\rangle \in ]-\infty, y_j[
\end{cases}$$

\newpage
\bibliographystyle{plain}
\bibliography{ref}

\end{document}